\documentclass[11pt,a4paper]{article}

\usepackage[left=2.3cm, top=2.3cm,bottom=2.3cm,right=2.3cm]{geometry}

\usepackage{mathtools,amsmath,amssymb,amsthm,mathrsfs,calc,graphicx,stmaryrd,xcolor,dsfont,tikz,pgfplots,bbm,float,array,epsfig,hyperref,color}
\usetikzlibrary{calc}
\usepackage[numbers,square]{natbib}
\usepackage[british]{babel}
\usepackage{amsfonts}              
\usepackage{enumitem}
\numberwithin{equation}{section}
\numberwithin{figure}{section}

\allowdisplaybreaks[4]

\def\sP{\mathscr{P}}
\def\sS{\mathcal{S}}

\newtheorem {theorem}{Theorem}[section]
\newtheorem {proposition}[theorem]{Proposition}
\newtheorem {lemma}[theorem]{Lemma}
\newtheorem {corollary}[theorem]{Corollary}

{\theoremstyle{definition}
\newtheorem{definition}[theorem]{Definition}

\newtheorem*{convention*}{Convention}

\newtheorem {example}[theorem]{Example}
\newtheorem {remark}[theorem]{Remark}
}

{
\theoremstyle{remark}
}

\newcommand{\Vol}{\operatorname{Vol}}

\renewcommand{\Re}{\operatorname{Re}}  
\renewcommand{\Im}{\operatorname{Im}}  

\newcommand{\ii}{{\rm{i}}}

\def\BB{\mathbb{B}}
\def\CC{\mathbb{C}}
\def\E{\mathbb{E}}

\def\N{\mathbb{N}}
\def\P{\mathbb{P}}

\def\R{\mathbb{R}}
\def\RRd1{\mathbb{R}^{d+1}}

\def\cI{\mathcal{I}}

\def\sG{\mathscr{G}}

\newcommand{\eee}{{\rm e}}

\newcommand{\ind}{\mathbbm{1}}
\newcommand{\eps}{\varepsilon}

\newcommand{\dd}{{\rm d}}
\newcommand{\bsl}{\backslash}

\begin{document}

\title{\bfseries Expected Hyperbolic Volumes of Random Beta Polytopes}

\author{Zakhar Kabluchko and Philipp Schange}

\date{}

\maketitle

\begin{abstract}
Let $X_1,\ldots,X_n$ be independent random points in the closed unit ball of $\mathbb{R}^d$. Assume that each $X_i$ has a beta distribution with parameter $\beta_i \ge -1$: if $\beta_i>-1$, then $X_i$ has Lebesgue density proportional to $(1-\|x\|^2)^{\beta_i}$ on $\{\|x\|<1\}$, whereas the case $\beta_i=-1$ corresponds to the uniform distribution on the unit sphere $\{\|x\|=1\}$. Let $[X_1,\ldots,X_n]$ denote the convex hull of these points. Interpreting the unit ball as the Klein model of hyperbolic geometry, we derive closed-form formulas for the expected hyperbolic volume of the random hyperbolic polytope $[X_1,\ldots,X_n]$. As a special case, if $X_1,\ldots,X_n$ are independent and uniformly distributed on the unit sphere in $\mathbb{R}^3$, then for every $n\ge 4$,
\[
\mathbb{E}\,\operatorname{Vol}_{3}^{\mathrm{hyp}}\!\bigl([X_1,\ldots,X_n]\bigr)
=
\pi\left(\frac{n}{2}-\sum_{j=1}^{n-1}\frac{1}{j}\right).
\]
\end{abstract}



\tableofcontents

\section{Introduction and selected  results}

\subsection{The Klein model}

In the Klein (or Beltrami--Klein) model of hyperbolic geometry, the underlying space is the open unit ball
$
\mathbb{B}^d := \{x \in \mathbb{R}^d : \|x\| < 1\},
$
equipped with the Riemannian metric
\[
\mathrm{d}s^2
=
\frac{(1-\|x\|^2)\,\|\mathrm{d}x\|^2 + \langle x,\mathrm{d}x\rangle^2}{(1-\|x\|^2)^2}.
\]
This metric has constant sectional curvature $-1$. Here, $\langle \cdot,\cdot\rangle$ denotes the standard Euclidean inner product on $\mathbb{R}^d$ with $d\geq 2$, and $\|\cdot\|$ the corresponding Euclidean norm. The ideal boundary (also called the absolute) is the unit sphere
$
\partial \mathbb{B}^d = \{x \in \mathbb{R}^d : \|x\| = 1\}.
$
We also write
$
\overline{\mathbb{B}}^d := \{x \in \mathbb{R}^d : \|x\| \le 1\}
$
for the closed unit ball.
The  hyperbolic volume of a Borel set $A \subseteq \overline{\mathbb{B}}^d$ is given by
\begin{equation}\label{eq:def_hyperbolic_volume}
\Vol_d^{\mathrm{hyp}}(A)
=
\int_A \frac{\dd x}{(1-\|x\|^2)^{(d+1)/2}}\in [0,+\infty].
\end{equation}
For further background on hyperbolic geometry, we refer to the book of \citet{ratcliffe_book}.

A convenient feature of the Klein model is that totally geodesic hyperbolic subspaces are represented by intersections of ordinary affine subspaces of $\mathbb{R}^d$ with $\mathbb{B}^d$. In particular, hyperbolic segments are precisely the Euclidean line segments contained in $\mathbb{B}^d$. As a consequence, hyperbolic convexity coincides with ordinary Euclidean convexity in the Klein model. In particular, a hyperbolic polytope is an ordinary convex polytope
$P = [v_1,\ldots,v_n]$
 contained in $\overline{\mathbb{B}}^d$, where $v_1,\ldots,v_n \in \overline{\mathbb{B}}^d$ and
$$
[v_1,\ldots, v_n]:= \{\lambda_1 v_1 + \ldots + \lambda_n v_n: \, \lambda_1\geq 0,\ldots, \lambda_n \geq 0, \lambda_1+\ldots+\lambda_n = 1\}
$$
denotes the convex hull. We refer to~\cite{Brondsted1983} and~\cite{Grunbaum2003} for background information on convex polytopes.  Some, or even all, vertices of the polytope $P$ may lie on the unit sphere $\partial \mathbb{B}^d$; such vertices are called \emph{ideal}. If all vertices of a hyperbolic polytope lie on $\partial \mathbb{B}^d$, then the polytope is called \emph{ideal}. The hyperbolic volume of a polytope $P$, denoted by $\Vol_d^{\mathrm{hyp}}(P)$, is defined by~\eqref{eq:def_hyperbolic_volume}. It is well known---and we shall sketch the argument below---that this volume is always finite, even if some or all vertices of $P$ are ideal.

Computing hyperbolic volumes of polytopes in dimensions greater than $2$ is, in general, a difficult problem going back to Lobachevsky, Bolyai, and Schl\"afli; see, for example, the papers of~\citet{milnor_hyperbolic_volumes,milnor_schlafli}, \citet{vinberg_volumes_polyhedra_russ_math_surv}, \citet{abrosimov_mednykh_volumes_polytopes_spaces_const_curv,abrosimov_mednykh}, and \citet{kellerhals_dilogs_vol_hyperbolic_polytopes}, as well as the books by~\citet[Chapter~7]{alekseevski_vinberg_solodovnikov_book} and~\citet[Chapters~10--11]{ratcliffe_book}. As one example, a closed formula for the hyperbolic volume of regular hyperbolic simplices can be found in~\cite{KabluchkoSchange2025Hyperbolic}.

\subsection{Selected results}
The aim of the present paper is to derive  closed-form expressions for the expected hyperbolic volume of random beta polytopes. This will be done in Theorems~\ref{theo:expected_hyp_volume_beta_polytopes_odd_d} and~\ref{theo:expected_hyp_volume_beta_polytopes_even_d} below. Since stating these results in full generality requires somewhat cumbersome notation, we begin with two special cases for which particularly simple formulas are available.

Our first result concerns random ideal polytopes in dimension $3$, whose vertices are independent and uniformly distributed on the unit sphere in $\mathbb{R}^3$.

\begin{theorem}[Expected hyperbolic volume of a random ideal polytope in dimension $3$]\label{theo:intro_dim_3_ideal_polytopes}
For $n \ge 4$, let $X_1, \ldots, X_n$ be independent random points uniformly distributed on the unit sphere $\partial \mathbb{B}^3 \subset \mathbb{R}^3$. Then the expected hyperbolic volume of the random polytope $\sP := [X_1,\ldots,X_n]$ is given by
\begin{align}\label{eq:exp_hyp_vol_random_ideal_polytope_dim_3_intro}
\E \Vol_3^{\mathrm{hyp}}(\sP)
=
\pi \left( \frac{n}{2} - \sum_{j=1}^{n-1} \frac{1}{j} \right).
\end{align}
\end{theorem}

For comparison, the analogous problem in dimension $d=2$ has a simple solution: the hyperbolic area of every ideal hyperbolic $n$-gon, that is, every Euclidean $n$-gon whose vertices lie on the unit circle, is equal to $(n-2)\pi$ by the Gauss--Bonnet formula. In higher dimensions, by contrast, no comparably simple general formulas for hyperbolic volumes are known.

In the next result we consider random ideal simplices in arbitrary dimension $d \ge 2$, obtained as the convex hull of $d+1$ independent random points uniformly distributed on $\partial \mathbb{B}^d$. Although the edge lengths of these simplices are infinite in the hyperbolic metric, their hyperbolic volume is finite and is bounded above by the volume of the \emph{regular} ideal simplex; see~\cite{haagerup_munkholm_simplices_max_vol_hyperbolic}.

\begin{theorem}[Expected hyperbolic volume of a random ideal simplex]\label{theo:intro_ideal_simplex_any_d}
Let $d \ge 2$ be an integer, and let $X_1,\ldots,X_{d+1}$ be independent random points uniformly distributed on the unit sphere $\partial \mathbb{B}^d \subset \mathbb{R}^d$. Consider the random ideal simplex $\sS := [X_1,\ldots,X_{d+1}]$.
\begin{itemize}
\item[(i)] If $d$ is \textbf{odd}, then
\[
\mathbb{E}\Vol_d^{\mathrm{hyp}}(\sS)
=
\frac{
2\pi^{\frac{d-1}{2}}
}{
\left(\frac{d-1}{2}\right)!
\binom{d^2-d-2}{\frac{d^2-d-2}{2}}
}
\int_0^1
t^{\frac{d-1}{2}}
\left(\sum_{j=0}^{\frac{d-3}{2}}
(-1)^j
\binom{d-2}{\frac{d-3}{2}-j}
\, t^j\right)^{d+1}
\, \mathrm{d}t.
\]
\item[(ii)] If $d$ is \textbf{even}, then
\[
\E \Vol_d^{\mathrm{hyp}}(\sS)
=
\frac{2^{1+d/2}}{\pi(d-1)!!}
\frac{\Gamma\left(\frac d2\right)^{d+1}}{\Gamma\left(\frac{d-1}2\right)^{d+1}}
\frac{\Gamma\left(\frac{d(d-1)}2\right)}{\Gamma\left(\frac{d(d-1)-1}2\right)}
\int_{-\infty}^{\infty}
\frac{\left(\int_0^t \sinh^{d-2}(u)\,\mathrm{d}u\right)^{d+1}}
{(\sinh t)^{d(d-1)-1}}\,\mathrm{d}t.
\]
\end{itemize}
\end{theorem}

For example, if $X_1,\ldots,X_4$ are independent random points uniformly distributed on the unit sphere in $\mathbb{R}^3$, then the expected hyperbolic volume of the random ideal tetrahedron $[X_1,\ldots,X_4]$ is equal to $\pi/6$. This follows either from Theorem~\ref{theo:intro_dim_3_ideal_polytopes} with $n=4$ or from Theorem~\ref{theo:intro_ideal_simplex_any_d} with $d=3$.

\section{Beta polytopes and associated special functions}

The aim of the present paper is to derive explicit formulas for the expected hyperbolic volume of random beta polytopes. In this section, we define beta distributions and beta polytopes, and introduce several special functions associated with them.

\subsection{Beta distributions}

We say that a random point $X$ in $\mathbb{B}^d$ has a \emph{beta distribution} with parameter $\beta > -1$ if its Lebesgue density is given by
\begin{equation}\label{eq:beta_density}
f_{d,\beta}(x)
=
c_{d,\beta} \left(1-\|x\|^2\right)^{\beta},
\quad x \in \mathbb{B}^d,
\qquad \text{where} \qquad
c_{d,\beta}
=
\frac{\Gamma\!\left(\frac d2 + \beta + 1\right)}{\pi^{d/2}\Gamma(\beta+1)}.
\end{equation}
In this case, we write $X \sim f_{d,\beta}$.
By convention, if a random point $X$ is uniformly distributed on the unit sphere $\partial \mathbb{B}^d$, then we say that it has beta distribution with parameter $-1$ and write $X \sim f_{d,-1}$. With this convention, the family of beta distributions extends continuously to $\beta=-1$ in the weak sense: as $\beta \downarrow -1$, the beta distribution with density $f_{d,\beta}$ converges weakly to the uniform distribution on $\partial \mathbb{B}^d$; see, for example,~\cite[Example~2.5]{KabluchkoSteigenbergerThale2026}.

Let us stress that~\eqref{eq:beta_density} is the density of the beta distribution with respect to Lebesgue measure on $\mathbb{R}^d$, and that $\|x\|$ denotes the Euclidean norm of $x$ (equivalently, the Euclidean distance from $x$ to the origin). Although this will not be needed in the sequel, let us also mention an intrinsic hyperbolic description of the beta distribution. Recall that the hyperbolic distance from $x \in \mathbb{B}^d$ to the origin $0$ is given by
\[
d_{\mathrm{hyp}}(0,x) = \operatorname{artanh}(\|x\|) = \frac{1}{2}\log \frac{1+\|x\|}{1-\|x\|}.
\]
Then, for $\beta > -1$, the beta distribution is the probability distribution on the $d$-dimensional hyperbolic space $\mathbb{B}^d$ whose density with respect to hyperbolic volume is given by
\[
x \mapsto c_{d,\beta}\cosh^{-(2\beta+d+1)}\bigl(d_{\mathrm{hyp}}(0,x)\bigr),
\qquad x \in \mathbb{B}^d.
\]
In particular, this density depends only on the hyperbolic distance from $x$ to the origin.

\subsection{Beta polytopes and beta integrals}
Let $X_1,\ldots,X_n$ be independent random points in $\overline{\BB}^d$ with distributions
$X_i \sim f_{d,\beta_i}$, where $\beta_1,\ldots,\beta_n \ge -1$.
Consider the random \emph{beta polytope}
\[
\sP := \sP_{n,d}^{\beta_1,\ldots,\beta_n} := [X_1,\ldots,X_n],
\]
where $[\cdot]$ denotes the Euclidean convex hull, equivalently, the hyperbolic convex hull in the Klein model.
We allow some, or even all, of the parameters $\beta_i$ to be equal to $-1$; in this case, $\sP$ has ideal vertices. Results on beta polytopes, viewed as Euclidean polytopes, were surveyed in~\cite[Chapter~8]{KabluchkoSteigenbergerThale2026}. In particular, the expected Euclidean volume, and more generally all expected intrinsic volumes, were determined in~\cite{KabluchkoTemesvariThale2019}. The expected number of faces in every dimension was determined in~\cite{KabluchkoThaleZaporozhets2020} and~\cite{Kabluchko2021_angles_simplices_face_numbers} in the case where all parameters $\beta_i$ are equal, and subsequently in~\cite{kabluchko_steigenberger_beta_cones} in the general case.

The aim of the present paper is to derive a closed-form expression for $\E \Vol_d^{\mathrm{hyp}}(\sP)$, the expected \emph{hyperbolic} volume of $\sP$. More generally, we shall determine the expected beta integral of $\sP$. For $\beta \in \mathbb{C}$, the \emph{beta integral} of a Borel set $A \subset \overline{\mathbb{B}}^d$ is defined by
\[
\int_A (1-\|x\|^2)^\beta \,\dd x,
\]
provided that the integral converges absolutely. Note that $\beta=0$ yields the usual Euclidean volume, whereas the hyperbolic volume is recovered by taking $\beta = -\frac{d+1}{2}$.

In Theorems~\ref{theo:expected_beta_volume_beta_polytopes} and~\ref{theo:expected_hyp_volume_beta_polytopes_at_poles}, we shall derive closed-form expressions for the expected beta integral of the random beta polytope. The expected hyperbolic volume will then appear as a special case.

Let us describe the first step of our approach. To express the expected beta integral of $\sP$, we introduce an additional random point $X_0 \sim f_{d,\beta}$, independent of $X_1 \sim f_{d,\beta_1},\ldots,X_n \sim f_{d,\beta_n}$. For the moment, this requires $\beta > -1$ to be real. Then
$$
\E \left[\int_{\sP}(1-\|x\|^2)^\beta \,\dd x \right]
=
c_{d,\beta}^{-1}\,\P\bigl[X_0 \in [X_1,\ldots,X_n]\bigr].
$$
Indeed, this identity is an immediate consequence of Fubini's theorem and the definition of the density $f_{d,\beta}$.
For the probability on the right-hand side, there is a formula from~\cite[Theorem~6.8]{kabluchko_steigenberger_beta_cones}, expressed in terms of certain special functions denoted by $a$ and $b$, which will be introduced below; we shall recall this formula in Theorem~\ref{theo:beta_absorption_beta_poly}. The geometric argument leading to this formula applies only when $\beta > -1$, because otherwise the distribution of $X_0$ is not well defined. In the present paper, we extend the domain of validity of this formula to all $\beta \in \mathbb{C}$ satisfying $\Re \beta > -\frac{d+1}{2}$ by an application of the uniqueness theorem for analytic continuation. As it turns out, the resulting formula has removable singularities at $\beta=-1,-2,\ldots$, which leads to additional technical difficulties. Finally, the expected hyperbolic volume is recovered by letting $\beta \downarrow -\frac{d+1}{2}$.

\subsection{The functions \texorpdfstring{$a$}{a}, \texorpdfstring{$b$}{b}, and \texorpdfstring{$\Theta$}{Theta}}\label{sec:def_of_a_b}
Our formulas for expected hyperbolic volumes and beta integrals of beta polytopes will be expressed in terms of the special functions $a$, $b$, and $\Theta$, which we define in this section. These functions were introduced in~\cite{kabluchko_steigenberger_beta_cones}, where they were used to express several expected geometric functionals of beta polytopes and beta cones, such as the expected $f$-vector.

To begin with, the special case of the function $c_{d,\beta}$ defined in~\eqref{eq:beta_density} corresponding to $d=1$ will be denoted by
\begin{equation}\label{eq:c_and_c_tilde}
c_\beta
\coloneqq
c_{1,\beta}
=
\frac{\Gamma\!\left(\beta + \frac32\right)}{\sqrt{\pi}\,\Gamma(\beta+1)}.
\end{equation}
Note that both $c_{d,\beta}$ and $c_\beta$ are meromorphic functions of $\beta \in \mathbb{C}$. We shall need the identities
\begin{equation}\label{eq:c_beta_integrals}
\int_{-\pi/2}^{\pi/2}\cos^\beta(z)\,\mathrm{d}z
=
\frac{1}{c_{(\beta-1)/2}}
\quad (\Re\beta > -1),
\qquad
\int_{-\infty}^{\infty}\frac{1}{\cosh^\beta(z)}\,\mathrm{d}z
=
\frac{1}{c_{\beta/2-1}}
\quad (\Re\beta > 0).
\end{equation}

\begin{definition}\label{def:F}
For $\beta \in \mathbb{C}$ with $\Re \beta > -1$, we define
\begin{align}\label{eq:def_F_as_cos_integral}
F_\beta(x)
\coloneqq
\int_{-\pi/2}^{\,x} \cos^\beta(y)\,\mathrm{d}y,
\qquad
x \in \sG
:=
\mathbb{C}\setminus\bigl((-\infty,-\pi/2]\cup[\pi/2,+\infty)\bigr).
\end{align}
\end{definition}
Here and in what follows, $\cos^\beta z$ denotes the branch that is holomorphic on $\sG$ and agrees with the usual power of a positive real number  for $z \in (-\pi/2,\pi/2)$. If $\beta$ is not an integer, then the function $z \mapsto \cos^\beta z$ has branch points at $\pi/2+\pi m$, $m \in \mathbb{Z}$. The slit domain $\sG$ is simply connected and avoids all these branch points. The integral in~\eqref{eq:def_F_as_cos_integral} is understood as a contour integral along any path in $\sG$ from $-\pi/2$ to $x$. Since $\Re \beta > -1$, the singularity at $-\pi/2$ is integrable, and hence $F_\beta$ is a well-defined single-valued holomorphic function on $\sG$. Since $\Re \beta > -1$, the function $F_\beta$ admits finite limits at the points $\pm \pi/2$ along the real axis. We therefore extend the above definition by setting
\[
F_\beta\!\left(-\frac{\pi}{2}\right) \coloneqq 0,
\qquad
F_\beta\!\left(\frac{\pi}{2}\right) \coloneqq \int_{-\pi/2}^{\pi/2} \cos^\beta(y)\,\mathrm{d}y = \frac{1}{c_{(\beta-1)/2}}.
\]
Using~\eqref{eq:c_beta_integrals}, we obtain, for $x \in \mathbb{R}$,
\begin{equation}\label{eq:F_a_imaginary_argument}
F_\beta(\ii x)
=
\frac{1}{2c_{(\beta-1)/2}} + \int_0^{\ii x} \cos^\beta(z)\,\mathrm{d}z
=
\frac{1}{2c_{(\beta-1)/2}} + \ii \int_0^x \cosh^\beta(y)\,\mathrm{d}y.
\end{equation}

\begin{definition}[$a$- and $b$-quantities]\label{def:a_and_b_quantities}
Let $d \in \mathbb{N}_0$, and let $\alpha_1,\ldots,\alpha_d \ge 0$ be real numbers.
\begin{itemize}
\item[(a)] For $\alpha \in \mathbb{C}$ with $\Re \alpha > \alpha_1+\cdots+\alpha_d$, we define
\begin{equation}\label{eq:def_a}
a(\alpha;\alpha_1,\ldots,\alpha_d)
\coloneqq
\int_{-\infty}^{\infty} \cosh^{-\alpha}(x)\prod_{j=1}^d F_{\alpha_j}(\ii x)\,\mathrm{d}x.
\end{equation}

\item[(b)] For $\alpha \in \mathbb{C}$ with $\Re \alpha > -1$, we define
\begin{equation}\label{eq:def_b}
b(\alpha;\alpha_1,\ldots,\alpha_d)
\coloneqq
\int_{-\pi/2}^{\pi/2} \cos^\alpha(x)\prod_{j=1}^d F_{\alpha_j}(x)\,\mathrm{d}x.
\end{equation}
\end{itemize}
\end{definition}

We shall use the following shorthand notation in the special case where the parameters $\alpha_1,\ldots,\alpha_d$ are all equal:
\[
a_d(\alpha;\beta)
:=
a(\alpha;\underbrace{\beta,\ldots,\beta}_{d\text{ times}}),
\qquad
b_d(\alpha;\beta)
:=
b(\alpha;\underbrace{\beta,\ldots,\beta}_{d\text{ times}}),
\qquad
d \in \mathbb{N}_0.
\]

Recall that a \emph{multiset} is a collection in which repetitions are allowed. The number of elements of a multiset $\Lambda=\{\lambda_1,\ldots,\lambda_d\}$, counted with multiplicities, is denoted by $|\Lambda|=d$. For such a multiset, we write
\[
a(\alpha;\Lambda) := a(\alpha;\lambda_1,\ldots,\lambda_d),
\qquad
b(\alpha;\Lambda) := b(\alpha;\lambda_1,\ldots,\lambda_d).
\]

\begin{lemma}[$a$ and $b$ are holomorphic]\label{lem:a_b_analytic}
Let $d \in \mathbb{N}_0$, and let $\alpha_1,\ldots,\alpha_d \ge 0$ be real numbers. Then the integral defining $a(\alpha;\alpha_1,\ldots,\alpha_d)$ converges absolutely for $\Re \alpha > \alpha_1+\cdots+\alpha_d$, and the integral defining $b(\alpha;\alpha_1,\ldots,\alpha_d)$ converges absolutely for $\Re \alpha > -1$. In these domains, the functions $a(\,\cdot\,;\alpha_1,\ldots,\alpha_d)$ and $b(\,\cdot\,;\alpha_1,\ldots,\alpha_d)$ are holomorphic.
\end{lemma}

\begin{proof} To prove analyticity, we apply a standard theorem on holomorphic parameter-dependent integrals; see~\cite[Kapitel~IV, Satz~5.8]{elstrodt_book} or~\cite[Chapter~XV, \S1, Lemma~1.1 on p.~409]{lang_book_complex_analysis}.

\vspace*{2mm}
\noindent
\textit{Proof for the function $b$.} For $x\in[-\pi/2,\pi/2]$, each $F_{\alpha_j}(x)$ is continuous, hence bounded, so there exists a constant $M>0$ such that
\[
\left|\prod_{j=1}^d F_{\alpha_j}(x)\right|\le M,
\qquad x\in\left[-\frac\pi2,\frac\pi2\right].
\]
Now let $K$ be a compact subset of the half-plane $\{\Re\alpha>-1\}$, and put $\sigma:=\inf_{\alpha\in K}\Re\alpha>-1$.
Then, for every $x\in(-\pi/2,\pi/2)$,
\[
\sup_{\alpha\in K} \left|\cos^\alpha(x)\prod_{j=1}^d F_{\alpha_j}(x)\right|
\le
M\,\cos^\sigma(x).
\]
Since $\sigma>-1$, the function $x\mapsto \cos^\sigma(x)$ is integrable on $(-\pi/2,\pi/2)$. Therefore,  the integral in~\eqref{eq:def_b} converges absolutely for every $\Re\alpha>-1$. Next, for fixed $x\in(-\pi/2,\pi/2)$, the map
$
\alpha\mapsto \cos^\alpha(x)\prod_{j=1}^d F_{\alpha_j}(x)
$
is entire.  By the standard theorem on holomorphic parameter-dependent integrals, the function  $\alpha\mapsto b(\alpha;\alpha_1,\ldots,\alpha_d)$ is holomorphic on the half-plane $\{\Re\alpha>-1\}$.

\vspace*{2mm}
\noindent
\textit{Proof for the function $a$.}
By~\eqref{eq:F_a_imaginary_argument}, for every $\beta\ge 0$  there is a constant $C_\beta>0$ such that
\[
|F_\beta(\ii x)|
\le
\frac{1}{2c_{(\beta-1)/2}}+\int_0^{|x|}\cosh^\beta(y)\,\mathrm dy
\le
C_\beta(1+|x|)\cosh^\beta(x), \qquad x\in\mathbb R.
\]
Hence there exists a constant $C>0$ such that
\[
\left|\prod_{j=1}^d F_{\alpha_j}(\ii x)\right|
\le
C(1+|x|)^d\cosh^{\alpha_1+\cdots+\alpha_d}(x),
\qquad x\in\mathbb R.
\]
Let now $K$ be a compact subset of the half-plane
$
\{\Re\alpha>\alpha_1+\cdots+\alpha_d\},
$
and put
$
\tau:=\inf_{\alpha\in K}\bigl(\Re\alpha-(\alpha_1+\cdots+\alpha_d)\bigr)>0.
$
Then, for  every $x\in\mathbb R$,
\[
\sup_{\alpha\in K} \left|\cosh^{-\alpha}(x)\prod_{j=1}^d F_{\alpha_j}(\ii x)\right|
\le
C(1+|x|)^d\cosh^{-\tau}(x).
\]
Since $\tau>0$, the function on the right-hand side is integrable over $\mathbb R$. Hence the integral in~\eqref{eq:def_a} converges absolutely. Moreover, for fixed $x\in\mathbb R$, the map
$
\alpha\mapsto \cosh^{-\alpha}(x)\prod_{j=1}^d F_{\alpha_j}(\ii x)
$
is entire. By the theorem on holomorphic parameter-dependent integrals, the function
$
\alpha\mapsto a(\alpha;\alpha_1,\ldots,\alpha_d)
$
is holomorphic on the half-plane $\{\Re\alpha>\alpha_1+\cdots+\alpha_d\}$.
\end{proof}

\begin{remark}[$d=0,1$]\label{rem:a_b_for_d_equal_0_1}
For $d=0$ and $d=1$, the integrals in the definitions of the functions $a$ and $b$ can be evaluated explicitly; see~\cite[Examples~4.5 and~4.10]{kabluchko_steigenberger_beta_cones}. We record these formulas for later use.
For the identities involving $a$, assume that $\alpha \in \CC$ satisfies $\Re \alpha > 0$ in the case $d=0$, and that $\alpha_1\geq 0$ and $\alpha \in \CC$ satisfy $\Re \alpha > \alpha_1$ in the case $d=1$. Then
\begin{equation}\label{eq:a_for_d_0_and_1}
a(\alpha;\varnothing)
=
c_{\frac{\alpha-2}{2}}^{-1}
=
\frac{\sqrt{\pi}\,\Gamma\!\left(\frac{\alpha}{2}\right)}{\Gamma\!\left(\frac{\alpha+1}{2}\right)}
\qquad \text{and} \qquad
a(\alpha;\alpha_1)
=
\frac{1}{2c_{\frac{\alpha_1-1}{2}}c_{\frac{\alpha-2}{2}}}
=
\frac{\pi\,\Gamma\!\left(\frac{\alpha}{2}\right)\Gamma\!\left(\frac{\alpha_1+1}{2}\right)}
{2\,\Gamma\!\left(\frac{\alpha+1}{2}\right)\Gamma\!\left(\frac{\alpha_1+2}{2}\right)}.
\end{equation}
For the identities involving $b$, assume that $\alpha \in \CC$ satisfies $\Re \alpha > -1$ in the case $d=0$, and that $\alpha_1\geq 0$ and $\alpha \in \CC$ satisfy $\Re \alpha > -1$ in the case $d=1$. Then
\begin{equation}\label{eq:b_for_d_0_and_1}
b(\alpha;\varnothing)
=
c_{\frac{\alpha-1}{2}}^{-1}
=
\frac{\sqrt{\pi}\,\Gamma\!\left(\frac{\alpha+1}{2}\right)}{\Gamma\!\left(\frac{\alpha+2}{2}\right)}
\qquad \text{and} \qquad
b(\alpha;\alpha_1)
=
\frac{1}{2c_{\frac{\alpha_1-1}{2}}c_{\frac{\alpha-1}{2}}}
=
\frac{\pi\,\Gamma\!\left(\frac{\alpha+1}{2}\right)\Gamma\!\left(\frac{\alpha_1+1}{2}\right)}
{2\,\Gamma\!\left(\frac{\alpha+2}{2}\right)\Gamma\!\left(\frac{\alpha_1+2}{2}\right)}.
\end{equation}
\end{remark}

\begin{remark}\label{rem:a_b_alternative_formula}
In~\cite[Propositions~4.6 and~4.11]{kabluchko_steigenberger_beta_cones}, several alternative representations for the functions $a$ and $b$ are given. We shall need the following one: for $d \in \N_0$, $\alpha_1,\ldots,\alpha_d \ge 0$, and $\alpha > -1$, one has
\begin{equation}\label{eq:b:formula_1}
b(\alpha;\alpha_1,\dots,\alpha_d)
=
\int_{-1}^{1} (1-t^2)^{\frac{\alpha-1}{2}}
\prod_{j=1}^{d}
\left(
\frac{1}{2c_{(\alpha_j-1)/2}}
+
\int_0^t (1-s^2)^{\frac{\alpha_j-1}{2}}\,\dd s
\right)\dd t.
\end{equation}
\end{remark}

In the formulas for expected functionals of beta polytopes and beta cones derived in~\cite{kabluchko_steigenberger_beta_cones}, the functions $a$ and $b$ appear only through a certain combination. It is therefore convenient to introduce the following additional notation.

\begin{definition}[$\Theta$-function]\label{def:theta}
Let $Y=\{y_1,\ldots,y_{\ell}\}$ and $Z=\{z_1,\ldots,z_k\}$ be finite multisets of non-negative numbers, and let $x\in \CC$ satisfy $\Re x > -1/2$. We define
\begin{multline}
\Theta(x;Y;Z)
\coloneqq
\frac{1}{2\pi}
\prod_{\omega \in Y} c_{\omega-\frac12}
\prod_{\omega \in Z} c_{\omega-\frac12}
\\
\times
a\left(2x+2\sum_{\omega \in Y}\omega+2;2Y\right)
\cdot
\left(2x+2\sum_{\omega \in Y}\omega+1\right)
\cdot
b\left(2x+2\sum_{\omega \in Y}\omega;2Z\right).
\end{multline}
In sums and products over a multiset, multiplicities are taken into account.
\end{definition}

The following result is taken from~\cite[Theorem~6.8]{kabluchko_steigenberger_beta_cones}. It will serve as the starting point of our proofs.

\begin{theorem}[Absorption of a beta point by a beta polytope]\label{theo:beta_absorption_beta_poly}
Let $d \ge 2$ and $n \ge d+1$. Let $X_0 \sim f_{d,\beta}$ and $X_i \sim f_{d,\beta_i}$, $i=1,\ldots,n$, be independent random points, where $\beta \ge -1$ and $\beta_1,\ldots,\beta_n \ge -1$. Then
\begin{align*}
\P[X_0 \in [X_1,\ldots,X_n]]
&=
2 \sum_{\substack{I\subseteq \{1,\ldots,n\} \\ |I|=d+1,d+3,\ldots}}
\Theta\left(
\beta+\frac d2;
\left\{\beta_i+\frac d2 : i\in I\right\};
\left\{\beta_i+\frac d2 : i\in I^c\right\}
\right)
\\
&=
1 - 2 \sum_{\substack{I\subseteq \{1,\ldots,n\} \\ |I|=d-1,d-3,\ldots}}
\Theta\left(
\beta+\frac d2;
\left\{\beta_i+\frac d2 : i\in I\right\};
\left\{\beta_i+\frac d2 : i\in I^c\right\}
\right),
\end{align*}
where $I^c=\{1,\ldots,n\}\setminus I$.
\end{theorem}

\section{Main results}

\subsection{Expected beta integrals of random beta polytopes}

The next theorem provides a closed-form expression for the expected beta integral of a random beta polytope. Note that the range of $\beta$ in this theorem is much larger than just $\beta \geq -1$.

\begin{theorem}[Expected beta integral of a random beta polytope]\label{theo:expected_beta_volume_beta_polytopes}
Let $d \ge 2$ and $n \ge d+1$ be integers, and let $X_1,\ldots,X_n$ be independent random points with distributions $X_i \sim f_{d,\beta_i}$, where $\beta_1,\ldots,\beta_n \ge -1$. Put $\gamma_i \coloneqq \beta_i + \frac d2$ for $i=1,\ldots,n$, and consider the random beta polytope $\sP \coloneqq [X_1,\ldots,X_n]$. Then, for every $\beta \in \CC$ with $\Re \beta > -\frac{d+1}{2}$ and $\beta \notin \{-1,-2,\ldots\}$, the expected beta integral
\begin{equation}\label{eq:I_beta_def_beta_vol}
\cI(\beta)
\coloneqq
\E\left[\int_{\sP} (1-\|x\|^2)^\beta \,\dd x \right]
\end{equation}
admits the representations
\begin{align}
\cI(\beta)
&=
\frac{2\pi^{d/2}\Gamma(\beta+1)}{\Gamma\!\left(\frac d2+\beta+1\right)}
\sum_{\substack{I\subseteq \{1,\ldots,n\} \\ |I|=d+1,d+3,\ldots}}
\Theta\left(
\beta+\frac d2;
\{\gamma_i : i\in I\};
\{\gamma_i : i\in I^c\}
\right)
\label{eq:beta_integral_for_hyp_vol_line_1}
\\
&=
\frac{\pi^{d/2-1}\Gamma(\beta+1)}{\Gamma\!\left(\frac d2+\beta+1\right)}
\left(
\prod_{i=1}^n \frac{\Gamma(\gamma_i+1)}{\sqrt{\pi}\,\Gamma\!\left(\gamma_i+\frac12\right)}
\right)
\sum_{\substack{I\subseteq \{1,\ldots,n\} \\ |I|=d+1,d+3,\ldots}}
\Bigg(
a\bigg(
2\beta+d+2+\sum_{i\in I}2\gamma_i;
\{2\gamma_i : i\in I\}
\bigg)
\notag
\\
&\qquad\qquad\qquad\qquad
\times
\bigg(
2\beta+d+1+\sum_{i\in I}2\gamma_i
\bigg)
b\bigg(
2\beta+d+\sum_{i\in I}2\gamma_i;
\{2\gamma_i : i\in I^c\}
\bigg)
\Bigg)
\label{eq:beta_integral_for_hyp_vol_line_2}
\end{align}
as well as
\begin{align}
\cI(\beta)
&=
\frac{2\pi^{d/2}\Gamma(\beta+1)}{\Gamma\!\left(\frac d2+\beta+1\right)}
\left(
\frac12
-
\sum_{\substack{I\subseteq \{1,\ldots,n\} \\ |I|=d-1,d-3,\ldots}}
\Theta\left(
\beta+\frac d2;
\{\gamma_i : i\in I\};
\{\gamma_i : i\in I^c\}
\right)
\right)
\label{eq:beta_integral_for_hyp_vol_2_line_1}
\\
&=
\frac{\pi^{d/2-1}\Gamma(\beta+1)}{\Gamma\!\left(\frac d2+\beta+1\right)}
\Bigg(
\pi
-
\left(
\prod_{i=1}^n \frac{\Gamma(\gamma_i+1)}{\sqrt{\pi}\,\Gamma\!\left(\gamma_i+\frac12\right)}
\right)
\notag
\\
&\qquad\qquad
\times
\sum_{\substack{I\subseteq \{1,\ldots,n\} \\ |I|=d-1,d-3,\ldots}}
\Bigg(
a\bigg(
2\beta+d+2+\sum_{i\in I}2\gamma_i;
\{2\gamma_i : i\in I\}
\bigg)
\notag
\\
&\qquad\qquad\qquad\qquad
\times
\bigg(
2\beta+d+1+\sum_{i\in I}2\gamma_i
\bigg)
b\bigg(
2\beta+d+\sum_{i\in I}2\gamma_i;
\{2\gamma_i : i\in I^c\}
\bigg)
\Bigg)
\Bigg).
\label{eq:beta_integral_for_hyp_vol_2_line_2}
\end{align}
Moreover, the function $\cI(\beta)$ is holomorphic on the half-plane $\{\Re \beta > -\frac{d+1}{2}\}$. The right-hand sides in \eqref{eq:beta_integral_for_hyp_vol_line_1}--\eqref{eq:beta_integral_for_hyp_vol_2_line_2} define holomorphic functions on the punctured half-plane $\{\Re \beta > -\frac{d+1}{2}\}\setminus \{-1,-2,\ldots\}$,
and their singularities at $\beta=-1,-2,\ldots$ are removable. Hence, they extend uniquely to holomorphic functions on $\{\Re \beta > -\frac{d+1}{2}\}$, and these extensions coincide with $\cI(\beta)$.
\end{theorem}

\begin{proof}
\textit{Step 1.} The function $\cI(\beta)$ is well defined and finite for all $\beta \in \CC$ with $\Re \beta > -1$. Indeed, for every realization $\sP(\omega)$ of $\sP$, we have
\[
\int_{\sP(\omega)} \bigl| (1-\|x\|^2)^\beta \bigr| \,\dd x
=
\int_{\sP(\omega)} (1-\|x\|^2)^{\Re\beta}\,\dd x
\le
\int_{\mathbb{B}^d} (1-\|x\|^2)^{\Re\beta}\,\dd x.
\]
The integral on the right-hand side is finite whenever $\Re \beta > -1$, and it does not depend on the realization. Hence $\cI(\beta)$ is well defined and finite on the half-plane $\{\Re \beta > -1\}$.

\medskip
\noindent
\textit{Step 2. Proof for real $\beta > -1$.}
Let $\beta$ be real and satisfy $\beta > -1$. Introduce an additional random point $X_0 \sim f_{d,\beta}$, independent of $X_1,\ldots,X_n$. Then, by Tonelli's theorem and the definition of the beta density,
\[
\cI(\beta)
=
\frac{1}{c_{d,\beta}}\,\P\bigl[X_0 \in [X_1,\ldots,X_n]\bigr].
\]
This probability can be evaluated using Theorem~\ref{theo:beta_absorption_beta_poly}, and we obtain 
\begin{align*}
\cI(\beta)
&=
\frac{2}{c_{d,\beta}}
\sum_{\substack{I\subseteq \{1,\ldots,n\} \\ |I|=d+1,d+3,\ldots}}
\Theta\left(
\beta+\frac d2;
\left\{\beta_i+\frac d2 : i\in I\right\};
\left\{\beta_i+\frac d2 : i\in I^c\right\}
\right)
\\
&=
\frac{2}{c_{d,\beta}}
\left(
\frac12
-
\sum_{\substack{I\subseteq \{1,\ldots,n\} \\ |I|=d-1,d-3,\ldots}}
\Theta\left(
\beta+\frac d2;
\left\{\beta_i+\frac d2 : i\in I\right\};
\left\{\beta_i+\frac d2 : i\in I^c\right\}
\right)
\right).
\end{align*}
These are exactly the formulas in~\eqref{eq:beta_integral_for_hyp_vol_line_1} and~\eqref{eq:beta_integral_for_hyp_vol_2_line_1}. Substituting the definition of $\Theta$ from Definition~\ref{def:theta} yields~\eqref{eq:beta_integral_for_hyp_vol_line_2} and~\eqref{eq:beta_integral_for_hyp_vol_2_line_2}.

Note that this argument applies only for real $\beta>-1$, because only in this case does $f_{d,\beta}$ define a probability density on $\mathbb{B}^d$, and hence only in this case can one introduce an auxiliary random point $X_0 \sim f_{d,\beta}$. To extend the range of $\beta$, we shall use the uniqueness theorem for analytic continuation.

\medskip
\noindent
\textit{Step 3. The function $\cI(\beta)$ is well defined on $\left\{\Re \beta \ge -\frac{d+1}{2}\right\}$.}
We now show that $\cI(\beta)$ is well defined and finite for all $\beta \in \CC$ with $\Re \beta \ge -\frac{d+1}{2}$.
It is well known, see~\cite{haagerup_munkholm_simplices_max_vol_hyperbolic}, that there exists a finite constant $M(d)>0$ such that every hyperbolic simplex $S \subseteq \overline{\mathbb{B}}^d$ has hyperbolic volume at most $M(d)$, that is,
\[
\int_{S} (1-\|x\|^2)^{-(d+1)/2}\,\dd x
\le
M(d).
\]

Now fix a realization $\sP(\omega)$ of the random polytope $\sP$. By Carath\'eodory's theorem, every point of $\sP(\omega)=[X_1(\omega),\ldots,X_n(\omega)]$ belongs to the convex hull of $d+1$ of the points $X_1(\omega),\ldots,X_n(\omega)$, that is
\[
\sP(\omega)
\subseteq
\bigcup_{\substack{J\subseteq \{1,\ldots,n\}\\ |J|=d+1}}
[X_j(\omega):j\in J].
\]
Next, for every $\beta \in \CC$ with $\Re \beta \ge -\frac{d+1}{2}$,
\[
\int_{\sP(\omega)} \bigl|(1-\|x\|^2)^\beta\bigr|\,\dd x
=
\int_{\sP(\omega)} (1-\|x\|^2)^{\Re\beta}\,\dd x
\le
\int_{\sP(\omega)} (1-\|x\|^2)^{-(d+1)/2}\,\dd x.
\]
Using the above covering and summing over the simplices, we obtain
\[
\int_{\sP(\omega)} (1-\|x\|^2)^{-(d+1)/2}\,\dd x
\le
\sum_{\substack{J\subseteq \{1,\ldots,n\}\\ |J|=d+1}}
\int_{[X_j(\omega):j\in J]} (1-\|x\|^2)^{-(d+1)/2}\,\dd x
\le
\binom{n}{d+1}M(d).
\]
This deterministic upper bound is independent of the realization of $\sP$, and hence $\cI(\beta)$ is well defined and finite on $\left\{\Re \beta \ge -\frac{d+1}{2}\right\}$.

\medskip
\noindent
\textit{Step 4. The function $\cI(\beta)$ is holomorphic on $\{\Re \beta > -\frac{d+1}{2}\}$.}
We next prove that $\cI$ is holomorphic on the half-plane $\{\Re \beta > -\frac{d+1}{2}\}$. To this end, we apply a standard theorem on holomorphic parameter-dependent integrals; see~\cite[Kapitel~IV, Satz~5.8]{elstrodt_book} or~\cite[Chapter~XV, \S1, Lemma~1.1 on p.~409]{lang_book_complex_analysis}.

Let $(\Omega,\mathcal{F},\P)$ be the probability space on which $X_1,\ldots,X_n$ are defined, and recall that $\sP(\omega)=[X_1(\omega),\ldots,X_n(\omega)]$. For $\beta \in \CC$ with $\Re \beta > -\frac{d+1}{2}$, we may write
\[
\cI(\beta)
=
\int_{\Omega\times \mathbb{B}^d} F((\omega,x),\beta)\,\dd x\,\P(\dd\omega),
\qquad
F((\omega,x),\beta)
:=
(1-\|x\|^2)^\beta \ind_{\{x\in \sP(\omega)\}}.
\]
For every fixed $(\omega,x)\in \Omega\times \mathbb{B}^d$, the function $\beta \mapsto F((\omega,x),\beta)$ is holomorphic on $\CC$.
Next, define
\[
g:\Omega\times \mathbb{B}^d \to [0,\infty),
\qquad
g(\omega,x)
:=
(1-\|x\|^2)^{-(d+1)/2}\ind_{\{x\in \sP(\omega)\}}.
\]
By Step~3, the function $g$ is integrable, because
\[
\int_{\Omega\times \mathbb{B}^d} g(\omega,x)\,\P(\dd\omega)\, \dd x
=
\E\left[\int_{\sP} (1-\|x\|^2)^{-(d+1)/2}\,\dd x\right]
\le
\binom{n}{d+1}M(d).
\]
Moreover, for every $(\omega,x)\in \Omega\times \mathbb{B}^d$ and every $\beta \in \CC$ with $\Re \beta > -\frac{d+1}{2}$, we have
\[
|F((\omega,x),\beta)|
=
(1-\|x\|^2)^{\Re\beta}\ind_{\{x\in \sP(\omega)\}}
\le
(1-\|x\|^2)^{-(d+1)/2}\ind_{\{x\in \sP(\omega)\}}
=
g(\omega,x).
\]
Therefore, \cite[Kapitel~IV, Satz~5.8]{elstrodt_book} implies that $\cI$ is holomorphic on $\left\{\Re \beta > -\frac{d+1}{2}\right\}$.

\medskip
\noindent
\textit{Step 5. The right-hand sides are meromorphic on $\left\{\Re \beta > -\frac{d+1}{2}\right\}$.}
We claim that the right-hand sides of~\eqref{eq:beta_integral_for_hyp_vol_line_2} and~\eqref{eq:beta_integral_for_hyp_vol_2_line_2} define meromorphic functions of $\beta$ on the half-plane $\left\{\Re \beta > -\frac{d+1}{2}\right\}$, and are holomorphic there away from the points $\beta=-1,-2,\ldots$.

Recall from Definition~\ref{def:a_and_b_quantities} that the functions $a(\alpha;\alpha_1,\ldots,\alpha_m)$ and $b(\alpha;\alpha_1,\ldots,\alpha_m)$ are defined for real parameters $\alpha_1,\ldots,\alpha_m \ge 0$ and complex $\alpha$ satisfying, respectively, $\Re \alpha > \alpha_1+\cdots+\alpha_m$ and $\Re \alpha > -1$. In our situation, the parameters entering the $a$- and $b$-terms are of the form $2\gamma_i=2\beta_i+d$, which are non-negative because $d \ge 2$ and $\beta_i \ge -1$.
Next, let $\beta \in \CC$ satisfy $\Re \beta > -\frac{d+1}{2}$. For every subset $I \subseteq \{1,\ldots,n\}$, the first argument of the corresponding $a$-term is
$
2\beta+d+2+\sum_{i\in I}2\gamma_i,
$
while the sum of the remaining arguments equals $\sum_{i\in I}2\gamma_i$. Hence
\[
\Re\Big(2\beta+d+2+\sum_{i\in I}2\gamma_i\Big)
>
\sum_{i\in I}2\gamma_i,
\]
because $\Re(2\beta+d+2)>1$. Likewise, the first argument of the corresponding $b$-term is
$
2\beta+d+\sum_{i\in I}2\gamma_i,
$
whose real part is greater than $-1$. Therefore, all $a$- and $b$-terms appearing in~\eqref{eq:beta_integral_for_hyp_vol_line_2} and~\eqref{eq:beta_integral_for_hyp_vol_2_line_2} are well defined and, by Lemma~\ref{lem:a_b_analytic}, holomorphic functions of $\beta$ on the half-plane $\{\Re \beta > -\frac{d+1}{2}\}$.

Furthermore, the function $1/\Gamma\!\left(\frac d2+\beta+1\right)$ is entire, whereas $\Gamma(\beta+1)$ is meromorphic on $\CC$ and holomorphic away from $\beta=-1,-2,\ldots$. It follows that the right-hand sides of~\eqref{eq:beta_integral_for_hyp_vol_line_2} and~\eqref{eq:beta_integral_for_hyp_vol_2_line_2} are meromorphic on $\{\Re \beta > -\frac{d+1}{2}\}$ and holomorphic there away from $\beta=-1,-2,\ldots$.

\medskip
\noindent
\textit{Step 6. Completion of the proof.}
We have shown that $\cI(\beta)$, that is, the left-hand side of~\eqref{eq:beta_integral_for_hyp_vol_line_2}, is holomorphic on the half-plane $\{\Re \beta > -\frac{d+1}{2}\}$. On the other hand, by Step~5, the right-hand side of~\eqref{eq:beta_integral_for_hyp_vol_line_2} is meromorphic on the same half-plane and holomorphic away from the points $\beta=-1,-2,\ldots$. Moreover, by Step~2, the two sides agree for all real $\beta>-1$.

Hence, on the punctured half-plane $\{\Re \beta > -\frac{d+1}{2}\}\setminus\{-1,-2,\ldots\}$, both sides are holomorphic and coincide on the set $\{\beta \in \mathbb{R} : \beta>-1\}$. By the identity theorem for analytic functions, they therefore coincide on $\{\Re \beta > -\frac{d+1}{2}\}\setminus\{-1,-2,\ldots\}$. Since the left-hand side is holomorphic on the whole half-plane, it follows that all singularities of the right-hand side of~\eqref{eq:beta_integral_for_hyp_vol_line_2} at $\beta=-1,-2,\ldots$ are removable.

The same argument applies to~\eqref{eq:beta_integral_for_hyp_vol_2_line_2}: its right-hand side is meromorphic on $\{\Re \beta > -\frac{d+1}{2}\}$ and holomorphic away from $\beta=-1,-2,\ldots$, while by Step~2 it agrees with $\cI(\beta)$ for all real $\beta>-1$. Hence, by the identity theorem, both sides coincide on $\{\Re \beta > -\frac{d+1}{2}\}\setminus\{-1,-2,\ldots\}$, and the singularities at $\beta=-1,-2,\ldots$ are removable. This completes the proof.
\end{proof}

\begin{remark}[On the removable singularities]\label{rem:removable_singularities}
Let us give a direct explanation for why the singularities in~\eqref{eq:beta_integral_for_hyp_vol_line_2} are removable. The poles of $\Gamma(\beta+1)$ at $\beta=-1,-2,\ldots$ are cancelled by zeros of the corresponding $a$-terms.
Indeed, it is known from~\cite[Proposition~5.16]{kabluchko_steigenberger_beta_cones} that for every $m \ge 2$, every $\ell \in \mathbb{N}_0$ satisfying $m-1-2\ell \ge 1$, and all $\alpha_1,\ldots,\alpha_m \ge 0$, one has
\[
a\bigl(m-1-2\ell+\alpha_1+\cdots+\alpha_m;\alpha_1,\ldots,\alpha_m\bigr)=0.
\]
Now consider an $a$-term appearing in~\eqref{eq:beta_integral_for_hyp_vol_line_2}, namely
$
a(2\beta+d+2+\sum_{i\in I}2\gamma_i;\{2\gamma_i:i\in I\}).
$
Writing $\{\alpha_1,\ldots,\alpha_m\}=\{2\gamma_i:i\in I\}$, we have $m=|I|$, and this term can be rewritten as
\[
a\bigl(m-1-2\ell+\alpha_1+\cdots+\alpha_m;\alpha_1,\ldots,\alpha_m\bigr),
\qquad
\ell \coloneqq -\beta + \frac{|I|-d-3}{2}.
\]
Now let $\beta \in \{-1,-2,\ldots\}$ satisfy $\beta \ge -\frac{d+1}{2}$, and let $|I| \in \{d+1,d+3,\ldots\}$. Then $m=|I| \in \mathbb{N}$, $m \ge d+1 \ge 3$, $\ell \in \mathbb{N}_0$, and $m-1-2\ell = 2\beta+d+2 \ge 1$.
Hence the above vanishing result applies, and all these $a$-terms are equal to zero at such values of $\beta$.

Since each $a$-term is holomorphic in $\beta$, this zero is of at least first order. Therefore, it cancels the simple pole of $\Gamma(\beta+1)$ at $\beta=-1,-2,\ldots$. This explains directly why the singularities in~\eqref{eq:beta_integral_for_hyp_vol_line_2} are removable.
\end{remark}

\subsection{Expected beta integrals at removable singularities}

The next theorem treats the case in which the factor $\Gamma(\beta+1)$ has a pole, so that the formulas from Theorem~\ref{theo:expected_beta_volume_beta_polytopes} do not apply directly.

\begin{theorem}[Expected beta integrals at removable singularities]\label{theo:expected_hyp_volume_beta_polytopes_at_poles}
Let $d \ge 2$ and $n \ge d+1$ be integers, and let $X_1,\ldots,X_n$ be independent random points with distributions $X_i \sim f_{d,\beta_i}$, where $\beta_1,\ldots,\beta_n \ge -1$. Put $\gamma_i \coloneqq \beta_i + \frac d2$ for $i=1,\ldots,n$, define $\sP \coloneqq [X_1,\ldots,X_n]$, and set
\[
\cI(\beta)
\coloneqq
\E\left[\int_{\sP} (1-\|x\|^2)^\beta \,\dd x \right].
\]
Then, for every $\beta \in \{-1,-2,-3,\ldots\}$ satisfying $\beta > -\frac{d+1}{2}$, one has
\begin{multline}\label{eq:beta_integral_for_hyp_vol_poles}
\cI(\beta)
=
\frac{2\pi^{\frac d2-1}}{\Gamma\!\left(\beta+\frac d2+1\right)}
\frac{(-1)^{-\beta-1}}{(-\beta-1)!}
\left(
\prod_{i=1}^n
\frac{\Gamma(\gamma_i+1)}{\sqrt{\pi}\,\Gamma\!\left(\gamma_i+\frac12\right)}
\right)
\\
\times
\sum_{\substack{I\subseteq \{1,\ldots,n\} \\ |I|=d+1,d+3,\ldots}}
\Bigg(
a'\bigg(
2\beta+d+2+\sum_{i\in I}2\gamma_i;
\{2\gamma_i : i\in I\}
\bigg)
\bigg(
2\beta+d+1+\sum_{i\in I}2\gamma_i
\bigg)
\\
\times
b\bigg(
2\beta+d+\sum_{i\in I}2\gamma_i;
\{2\gamma_i : i\in I^c\}
\bigg)
\Bigg).
\end{multline}
Here, $a'$ denotes the derivative of $a$ with respect to its first argument. Alternatively,
\begin{multline}\label{eq:beta_integral_for_hyp_vol_poles_2_ab}
\cI(\beta)
=
-\frac{2\pi^{\frac d2-1}}{\Gamma\!\left(\beta+\frac d2+1\right)}
\frac{(-1)^{-\beta-1}}{(-\beta-1)!}
\left(
\prod_{i=1}^n
\frac{\Gamma(\gamma_i+1)}{\sqrt{\pi}\,\Gamma\!\left(\gamma_i+\frac12\right)}
\right)
\\
\times
\sum_{\substack{I\subseteq \{1,\ldots,n\} \\ |I|=d-1,d-3,\ldots}}
\left.
\frac{\partial}{\partial \alpha}
\Bigg(
a\bigg(\alpha+2;\{2\gamma_i : i\in I\}\bigg)
(\alpha+1)
b\bigg(\alpha;\{2\gamma_i : i\in I^c\}\bigg)
\Bigg)
\right|_{\alpha=\,2\beta+d+\sum_{i\in I}2\gamma_i}.
\end{multline}
Moreover, the parameters of the functions $a$ and $b$ appearing in~\eqref{eq:beta_integral_for_hyp_vol_poles} and~\eqref{eq:beta_integral_for_hyp_vol_poles_2_ab} satisfy the assumptions of Definition~\ref{def:a_and_b_quantities}. In particular, all terms on the right-hand sides of~\eqref{eq:beta_integral_for_hyp_vol_poles} and~\eqref{eq:beta_integral_for_hyp_vol_poles_2_ab} are finite.
\end{theorem}

\begin{proof}
\textit{Proof of~\eqref{eq:beta_integral_for_hyp_vol_poles}}. Fix $k \in \{1,2,3,\ldots\}$ with $k < \frac{d+1}{2}$, and set $\beta_\eps \coloneqq -k+\eps$ with $\eps>0$. Substituting $\beta=\beta_\eps$ into~\eqref{eq:beta_integral_for_hyp_vol_line_2}, we obtain
\begin{align}
\cI(\beta_\eps)
&=
\frac{\pi^{\frac d2-1}\Gamma(-k+1+\eps)}{\Gamma\!\left(\frac d2-k+1+\eps\right)}
\left(
\prod_{i=1}^n \frac{\Gamma(\gamma_i+1)}{\sqrt{\pi}\,\Gamma\!\left(\gamma_i+\frac12\right)}
\right)
\notag
\\
&\quad\times
\sum_{\substack{I\subseteq \{1,\ldots,n\} \\ |I|=d+1,d+3,\ldots}}
\Bigg(
a\bigg(
-2k+2\eps+d+2+\sum_{i\in I}2\gamma_i;
\{2\gamma_i:i\in I\}
\bigg)
\notag
\\
&\qquad\qquad\qquad\times
\bigg(
-2k+2\eps+d+1+\sum_{i\in I}2\gamma_i
\bigg)
b\bigg(
-2k+2\eps+d+\sum_{i\in I}2\gamma_i;
\{2\gamma_i:i\in I^c\}
\bigg)
\Bigg).
\label{eq:beta_integral_for_hyp_vol_line_2_with_eps}
\end{align}
We now let $\eps \downarrow 0$. Since $\cI$ is holomorphic at $-k$ by Theorem~\ref{theo:expected_beta_volume_beta_polytopes}, we have $\cI(-k+\eps)\to \cI(-k)$ as $\eps\downarrow 0$.
Next, the gamma factor has the expansion
\begin{equation}\label{eq:gamma_pole_residuum}
\Gamma(-k+1+\eps)
=
\frac{(-1)^{k-1}}{(k-1)!\,\eps}
+
O(1),
\qquad
\eps\downarrow 0.
\end{equation}
Moreover, $\Gamma\!\left(\frac d2-k+1+\eps\right)\to \Gamma\!\left(\frac d2-k+1\right)$, which is finite and non-zero because $k<\frac{d+1}{2}$.
Now fix a set $I$ occurring in the sum. Put
$
\alpha_I \coloneqq -2k+d+\sum_{i\in I}2\gamma_i.
$
Then the corresponding $b$-term is
$
b(\alpha_I+2\eps;\{2\gamma_i:i\in I^c\}).
$
Since $\alpha_I \ge d-2k>-1$, the function $b$ is holomorphic at $\alpha_I$ by Lemma~\ref{lem:a_b_analytic}, and therefore
\[
b\bigl(\alpha_I+2\eps;\{2\gamma_i:i\in I^c\}\bigr)
=
b\bigl(\alpha_I;\{2\gamma_i:i\in I^c\}\bigr)+o(1)
\qquad
\text{as } \eps\downarrow 0.
\]
By Remark~\ref{rem:removable_singularities}, the corresponding $a$-term vanishes at $\eps=0$. Since $a$ is holomorphic in its first argument, we obtain the first-order expansion
\[
a\bigl(\alpha_I+2+2\eps;\{2\gamma_i:i\in I\}\bigr)
=
2\eps\,
a'\bigl(\alpha_I+2;\{2\gamma_i:i\in I\}\bigr)
+
o(\eps)
\qquad
\text{as } \eps\downarrow 0.
\]
Substituting these expansions into~\eqref{eq:beta_integral_for_hyp_vol_line_2_with_eps}, and using that the sum is finite, we obtain exactly~\eqref{eq:beta_integral_for_hyp_vol_poles}.

\medskip \noindent
\textit{Proof of~\eqref{eq:beta_integral_for_hyp_vol_poles_2_ab}.} We argue similarly, now starting from~\eqref{eq:beta_integral_for_hyp_vol_2_line_2}. Put $\beta=-k+\eps$ with $\eps>0$ and let $\eps\downarrow 0$. As before, we have~\eqref{eq:gamma_pole_residuum}.
Since $\cI$ is holomorphic at $-k$, the bracket in~\eqref{eq:beta_integral_for_hyp_vol_2_line_2} must vanish at $\eps=0$. Hence its first-order Taylor expansion is
{\scriptsize
\begin{multline*}
\pi
-
\left( \prod_{i=1}^n \frac{\Gamma(\gamma_i+1)}{\sqrt{\pi}\,\Gamma\!\left(\gamma_i+\frac12\right)} \right)
\sum_{\substack{I\subseteq \{1,\ldots,n\}\\ |I|=d-1,d-3,\ldots}}
\Bigg(
a\bigg(-2k+2\eps+d+2+\sum_{i\in I}2\gamma_i;\{2\gamma_i:i\in I\}\bigg)
\\
\times
\bigg(-2k+2\eps+d+1+\sum_{i\in I}2\gamma_i\bigg)
\,b\bigg(-2k+2\eps+d+\sum_{i\in I}2\gamma_i;\{2\gamma_i:i\in I^c\}\bigg)
\Bigg)
\\
=
-2\eps
\left( \prod_{i=1}^n \frac{\Gamma(\gamma_i+1)}{\sqrt{\pi}\,\Gamma\!\left(\gamma_i+\frac12\right)} \right)
\!\!\!\!\!
\sum_{\substack{I\subseteq \{1,\ldots,n\}\\ |I|=d-1,d-3,\ldots}}
\!\!\!\!\!
\left.
\frac{\partial}{\partial \alpha}
\Bigg(
a\bigg(\alpha+2;\{2\gamma_i:i\in I\}\bigg)
(\alpha+1)
b\bigg(\alpha;\{2\gamma_i:i\in I^c\}\bigg)
\Bigg)
\right|_{\alpha=\, -2k+d+\sum_{i\in I}2\gamma_i}
+o(\eps).
\end{multline*}
}
Multiplying this expansion by~\eqref{eq:gamma_pole_residuum} and letting $\eps\downarrow 0$ yields~\eqref{eq:beta_integral_for_hyp_vol_poles_2_ab}.
\end{proof}

\subsection{Expected hyperbolic volume}

The expected hyperbolic volume of a beta polytope $\sP$ is given by
\[
\E \Vol_d^{\mathrm{hyp}}(\sP)
=
\E\left[\int_{\sP} (1-\|x\|^2)^{-\frac{d+1}{2}}\,\dd x\right].
\]
To compute it, we let $\beta \downarrow -\frac{d+1}{2}$ in Theorem~\ref{theo:expected_beta_volume_beta_polytopes}. The resulting formulas depend on the parity of $d$.

\begin{theorem}[Expected hyperbolic volume of beta polytopes: odd $d$]\label{theo:expected_hyp_volume_beta_polytopes_odd_d}
Let $d \ge 3$ be \textbf{odd}, let $n \ge d+1$, and let $X_1,\ldots,X_n$ be independent random points with distributions $X_i \sim f_{d,\beta_i}$, where $\beta_1,\ldots,\beta_n \ge -1$. Put $\gamma_i \coloneqq \beta_i + \frac d2$ for $i=1,\ldots,n$, and define $\sP \coloneqq [X_1,\ldots,X_n]$. Then
\begin{multline}\label{eq:beta_integral_for_hyp_vol_poles_hyperbolic_odd_d}
\E \Vol_d^{\mathrm{hyp}}(\sP)
=
\frac{2\pi^{\frac{d-3}{2}}(-1)^{\frac{d-1}{2}}}{\left(\frac{d-1}{2}\right)!}
\left(
\prod_{i=1}^n
\frac{\Gamma(\gamma_i+1)}{\sqrt{\pi}\,\Gamma\!\left(\gamma_i+\frac12\right)}
\right)
\\
\times
\sum_{\substack{I\subseteq \{1,\ldots,n\} \\ |I|=d+1,d+3,\ldots}}
a'\bigg(
1+\sum_{i\in I}2\gamma_i;
\{2\gamma_i:i\in I\}
\bigg)
\bigg(
\sum_{i\in I}2\gamma_i
\bigg)
b\bigg(
\sum_{i\in I}2\gamma_i-1;
\{2\gamma_i:i\in I^c\}
\bigg).
\end{multline}
The parameters of the functions $a$ and $b$ appearing in~\eqref{eq:beta_integral_for_hyp_vol_poles_hyperbolic_odd_d} satisfy the assumptions of Definition~\ref{def:a_and_b_quantities}. In particular, all terms on the right-hand side are finite.
\end{theorem}
\begin{proof}
The argument is almost the same as in the proof of Theorem~\ref{theo:expected_hyp_volume_beta_polytopes_at_poles}. We substitute $\beta=-\frac{d+1}{2}+\eps$ with $\eps>0$ into~\eqref{eq:beta_integral_for_hyp_vol_line_2} and let $\eps\downarrow 0$.
Since $d$ is odd,
\[
\Gamma\!\left(-\frac{d-1}{2}+\eps\right)
=
\frac{(-1)^{\frac{d-1}{2}}}{\left(\frac{d-1}{2}\right)!\,\eps}
+O(1),
\qquad
\Gamma\!\left(\frac12+\eps\right)\to \sqrt{\pi}.
\]
By Remark~\ref{rem:removable_singularities}, each $a$-term in~\eqref{eq:beta_integral_for_hyp_vol_line_2} has the expansion
\[
a\bigg(
1+2\eps+\sum_{i\in I}2\gamma_i;
\{2\gamma_i:i\in I\}
\bigg)
=
2\eps\,
a'\bigg(
1+\sum_{i\in I}2\gamma_i;
\{2\gamma_i:i\in I\}
\bigg)
+o(\eps).
\]
Moreover, $-1+\sum_{i\in I}2\gamma_i>-1$, because $I\neq\varnothing$ and $\gamma_i=\beta_i+\frac d2>0$ for all $i$, so the corresponding $b$-term in~\eqref{eq:beta_integral_for_hyp_vol_line_2} converges by continuity.

On the other hand, as $\eps\downarrow 0$, the functions $(1-\|x\|^2)^{-\frac{d+1}{2}+\eps}\ind_{\{x\in \sP(\omega)\}}$
increase pointwise to $(1-\|x\|^2)^{-\frac{d+1}{2}}\ind_{\{x\in \sP(\omega)\}}$,
and therefore the monotone convergence theorem yields
\[
\cI\!\left(-\frac{d+1}{2}+\eps\right)\to \E \Vol_d^{\mathrm{hyp}}(\sP).
\]
Combining these facts and passing to the limit in~\eqref{eq:beta_integral_for_hyp_vol_line_2} proves~\eqref{eq:beta_integral_for_hyp_vol_poles_hyperbolic_odd_d}.
\end{proof}

If $d$ is even, then the factor $\Gamma(\beta+1)$ in~\eqref{eq:beta_integral_for_hyp_vol_line_2} is holomorphic at $\beta=-\frac{d+1}{2}$. We shall show that, after a suitable interpretation of the potentially singular terms, the formulas of Theorem~\ref{theo:expected_beta_volume_beta_polytopes} remain valid at this value of $\beta$.

\begin{theorem}[Expected hyperbolic volume of beta polytopes: even $d$]\label{theo:expected_hyp_volume_beta_polytopes_even_d}
Let $d \ge 2$ be \textbf{even}, let $n \ge d+1$, and let $X_1,\ldots,X_n$ be independent random points with distributions $X_i \sim f_{d,\beta_i}$, where $\beta_1,\ldots,\beta_n \ge -1$. Put $\gamma_i \coloneqq \beta_i + \frac d2$ for $i=1,\ldots,n$, and define $\sP \coloneqq [X_1,\ldots,X_n]$. Then the formulas of Theorem~\ref{theo:expected_beta_volume_beta_polytopes} extend to $\beta=-\frac{d+1}{2}$, and in particular
\begin{align}
\E \Vol_d^{\mathrm{hyp}}(\sP)
&=
\frac{(-2\pi)^{d/2}}{\pi(d-1)!!}
\left(
\prod_{i=1}^n
\frac{\Gamma(\gamma_i+1)}{\sqrt{\pi}\,\Gamma\!\left(\gamma_i+\frac12\right)}
\right)
\notag
\\
&\hspace{-3em} \times\!\!\!\!\!
\sum_{\substack{I\subseteq \{1,\ldots,n\} \\ |I|=d+1,d+3,\ldots}}
a\bigg(
1+\sum_{i\in I}2\gamma_i;
\{2\gamma_i:i\in I\}
\bigg)
\bigg(
\sum_{i\in I}2\gamma_i
\bigg)
b\bigg(
\sum_{i\in I}2\gamma_i-1;
\{2\gamma_i:i\in I^c\}
\bigg)
\label{eq:beta_integral_for_hyp_vol_line_2_hyperbolic}
\end{align}
as well as
\begin{align}
\E \Vol_d^{\mathrm{hyp}}(\sP)
&=
\frac{(-2\pi)^{d/2}}{\pi(d-1)!!}
\Bigg(
\pi
-
\left(
\prod_{i=1}^n
\frac{\Gamma(\gamma_i+1)}{\sqrt{\pi}\,\Gamma\!\left(\gamma_i+\frac12\right)}
\right)
\notag
\\
&\hspace{-3em} \times\!\!\!\!\!
\sum_{\substack{I\subseteq \{1,\ldots,n\} \\ |I|=d-1,d-3,\ldots}}
a\bigg(
1+\sum_{i\in I}2\gamma_i;
\{2\gamma_i:i\in I\}
\bigg)
\bigg(
\sum_{i\in I}2\gamma_i
\bigg)
b\bigg(
\sum_{i\in I}2\gamma_i-1;
\{2\gamma_i:i\in I^c\}
\bigg)
\Bigg).
\label{eq:beta_integral_for_hyp_vol_2_line_2_hyperbolic}
\end{align}
Here, any indeterminate expression of the form
$
0 \cdot b(-1;\{2\gamma_i:i\in I^c\})
$
is understood in the limiting sense described in Lemma~\ref{lemma:limit_if_b_has_pole} below.
\end{theorem}

\begin{remark}
The right-hand sides of~\eqref{eq:beta_integral_for_hyp_vol_line_2_hyperbolic} and~\eqref{eq:beta_integral_for_hyp_vol_2_line_2_hyperbolic} may be interpreted as the values of the corresponding $\Theta$-expressions at the first argument $-\frac12$. With this convention,
\begin{align}
\E \Vol_d^{\mathrm{hyp}}(\sP)
&=
\frac{2(-2\pi)^{d/2}}{(d-1)!!}
\sum_{\substack{I\subseteq \{1,\ldots,n\} \\ |I|=d+1,d+3,\ldots}}
\Theta\left(
-\frac12;
\{\gamma_i:i\in I\};
\{\gamma_i:i\in I^c\}
\right)
\label{eq:beta_integral_for_hyp_vol_line_1_hyperbolic}
\\
&=
\frac{2(-2\pi)^{d/2}}{(d-1)!!}
\left(
\frac12
-
\sum_{\substack{I\subseteq \{1,\ldots,n\} \\ |I|=d-1,d-3,\ldots}}
\Theta\left(
-\frac12;
\{\gamma_i:i\in I\};
\{\gamma_i:i\in I^c\}
\right)
\right).
\label{eq:beta_integral_for_hyp_vol_2_line_1_hyperbolic}
\end{align}
\end{remark}

\begin{proof}[Proof of Theorem~\ref{theo:expected_hyp_volume_beta_polytopes_even_d}]
For every real $\beta>-\frac{d+1}{2}$, Equations~\eqref{eq:beta_integral_for_hyp_vol_line_2} and~\eqref{eq:beta_integral_for_hyp_vol_2_line_2} of Theorem~\ref{theo:expected_beta_volume_beta_polytopes} apply. We let $\beta\downarrow -\frac{d+1}{2}$ in these formulas. Recall that $\cI(\beta)$ is defined by~\eqref{eq:I_beta_def_beta_vol}. By monotone convergence,
\[
\E \Vol_d^{\mathrm{hyp}}(\sP)
=
\E\left[\int_{\sP}(1-\|x\|^2)^{-\frac{d+1}{2}}\,\dd x\right]
=
\lim_{\beta\downarrow -\frac{d+1}{2}} \cI(\beta).
\]
The prefactor in~\eqref{eq:beta_integral_for_hyp_vol_line_2} and~\eqref{eq:beta_integral_for_hyp_vol_2_line_2} has the limit
\[
\lim_{\beta\downarrow -\frac{d+1}{2}}
\frac{\pi^{\frac d2-1}\Gamma(\beta+1)}
{\Gamma\!\left(\frac d2+\beta+1\right)}
=
\frac{\pi^{\frac d2-1}\Gamma\!\left(\frac{1-d}{2}\right)}{\Gamma\!\left(\frac12\right)}
=
\frac{(-2\pi)^{d/2}}{\pi(d-1)!!}.
\]
Next, fix one summand in~\eqref{eq:beta_integral_for_hyp_vol_line_2}. As $\beta\downarrow -\frac{d+1}{2}$, the first argument of the $a$-function tends to $1+\sum_{i\in I}2\gamma_i$,
and since $1+\sum_{i\in I}2\gamma_i>\sum_{i\in I}2\gamma_i$, Lemma~\ref{lem:a_b_analytic} shows that
$$
a\bigg(
2\beta+d+2+\sum_{i\in I}2\gamma_i;
\{2\gamma_i : i\in I\}
\bigg)
\to
a\bigg(
1+\sum_{i\in I}2\gamma_i;
\{2\gamma_i : i\in I\}
\bigg).
$$
Now consider the factor involving $b$. If $\sum_{i\in I}2\gamma_i>0$, then $\sum_{i\in I}2\gamma_i-1>-1$, and Lemma~\ref{lem:a_b_analytic} implies that
$$
\bigg(2\beta+d+1+\sum_{i\in I}2\gamma_i\bigg)
b\bigg(2\beta+d+\sum_{i\in I}2\gamma_i;\{2\gamma_i:i\in I^c\}\bigg)
\to
\bigg(\sum_{i\in I}2\gamma_i\bigg)
b\bigg(\sum_{i\in I}2\gamma_i-1;\{2\gamma_i:i\in I^c\}\bigg).
$$
The only possible exceptional case is therefore $\sum_{i\in I}2\gamma_i=0$. Since $\gamma_i=\beta_i+\frac d2\ge 0$ and $I\neq \varnothing$, this can happen only if $\gamma_i=0$ for all $i\in I$, which is possible only when $d=2$ and $\beta_i=-1$ for all $i\in I$. In this case, the limit of the above factor  exists and is finite by Lemma~\ref{lemma:limit_if_b_has_pole}.

Thus every summand in~\eqref{eq:beta_integral_for_hyp_vol_line_2} has a finite limit as $\beta\downarrow -\frac{d+1}{2}$. Since the sum is finite, we may pass to the limit term by term, which yields~\eqref{eq:beta_integral_for_hyp_vol_line_2_hyperbolic}.

The proof of~\eqref{eq:beta_integral_for_hyp_vol_2_line_2_hyperbolic} is analogous, starting from~\eqref{eq:beta_integral_for_hyp_vol_2_line_2}. Note that the index set in~\eqref{eq:beta_integral_for_hyp_vol_2_line_2} does not contain $I=\varnothing$, because $d$ is even and hence the cardinalities $d-1,d-3,\ldots$ are positive odd integers.
\end{proof}

\begin{lemma}[Pole of $b$ at $\alpha=-1$]\label{lemma:limit_if_b_has_pole}
Let $d \in \mathbb{N}_0$, and let $\alpha_1,\ldots,\alpha_d \ge 0$. Then
\begin{equation}\label{eq:limit_if_b_has_pole}
\lim_{\alpha \downarrow -1} (\alpha+1)\,b(\alpha;\alpha_1,\ldots,\alpha_d)
=
\begin{cases}
2, & \text{if } d=0,\\[1mm]
\displaystyle
\prod_{i=1}^d \frac{1}{c_{\frac{\alpha_i-1}{2}}}
=
\prod_{i=1}^d
\frac{\sqrt{\pi}\,\Gamma\!\left(\frac{\alpha_i+1}{2}\right)}
{\Gamma\!\left(\frac{\alpha_i+2}{2}\right)},
& \text{if } d\ge 1.
\end{cases}
\end{equation}
\end{lemma}

\begin{proof}
For $d=0$, the explicit formula $b(\alpha;\varnothing) = c_{(\alpha-1)/2}^{-1}$ from~\eqref{eq:b_for_d_0_and_1} entails
\begin{align*}
\lim_{\alpha \downarrow -1} (\alpha +1) b(\alpha; \varnothing)
=
\lim_{\alpha \downarrow -1} (\alpha + 1) \frac{\sqrt \pi \Gamma \left( \frac{\alpha+1}{2} \right)}{\Gamma \left( \frac{\alpha+2}{2} \right)}
=
\lim_{\alpha \downarrow -1} 2 \frac{\sqrt \pi \Gamma \left( \frac{\alpha+1}{2}+1 \right)}{\Gamma \left( \frac{\alpha+2}{2} \right)}
=
2.
\end{align*}
Now, let $d\geq 1$. Fix $\alpha_1, \ldots, \alpha_d \geq 0$ and define
$$
g_i(t)
:=
\frac{1}{2c_{\frac{\alpha_i-1}2}}
+\int_0^t (1-s^2)^{\frac{\alpha_i-1}{2}}\dd s
=
\int_{-1}^t (1-s^2)^{\frac{\alpha_i-1}{2}}\dd s,
\qquad t\in[-1,1].
$$
The second equality holds because $\int_{-1}^0 (1-s^2)^{(\alpha_i-1)/2} \dd s = \frac 12 B(\frac 12,\frac{\alpha_i+1}2)$.
By~\eqref{eq:b:formula_1}, for $\alpha>-1$,
\begin{multline*}
b(\alpha;\alpha_1,\ldots,\alpha_d)
=
\int_{-1}^{1}(1-t^2)^{\frac{\alpha-1}{2}} \prod_{i=1}^d g_i(t) \dd t
=
\underbrace{\int_{-1}^{0}(1-t^2)^{\frac{\alpha-1}{2}} \prod_{i=1}^d g_i(t) \dd t}_{=: I_1(\alpha)}
\\
+
\underbrace{\int_{0}^{1}(1-t^2)^{\frac{\alpha-1}{2}} \left( \prod_{i=1}^d g_i(t) - \prod_{i=1}^d g_i(1) \right) \dd t}_{=: I_2(\alpha)}
+
\underbrace{\left( \prod_{i=1}^d g_i(1) \right) \int_{0}^{1}(1-t^2)^{\frac{\alpha-1}{2}} \dd t}_{=: I_3(\alpha)}.
\end{multline*}

Let us show that $I_1(\alpha)$ and $|I_2(\alpha)|$ stay bounded as $\alpha \downarrow -1$, so that only the third integral contributes to the limit~\eqref{eq:limit_if_b_has_pole}.
Using $(1-s^2) = (1-s)(1+s)$ yields that, for all $i=1,\ldots, d$, there exists $c>0$ such that
$$
g_i(t)
=
\int_{-1}^t (1-s^2)^{\frac{\alpha_i-1}2} \dd s
\leq
c \int_{-1}^t (1+s)^{\frac{\alpha_i-1}2} \dd s
=
\frac{2c}{\alpha_i+1} (1+t)^{\frac{\alpha_i+1}{2}}
\leq
\frac{2c}{\alpha_i+1} (1+t)^{\frac{1}{2}},
$$
for all $t\in [-1,0]$.
It follows that there exists $C_1>0$ such that for all $\alpha\in (-1,0)$, one has
\begin{align*}
I_1(\alpha)
\leq
C_1  \int_{-1}^0 (1+t)^{\frac{\alpha-1}{2}+\frac d2} \dd t
\leq
C_1  \int_{-1}^0 (1+t)^{\frac{d-2}2} \dd t
<\infty.
\end{align*}

Similarly, for $I_2(\alpha)$ we have $|g_i(t)-g_i(1)|\leq \frac{2c}{\alpha_i+1}  (1-t)^{\frac12}$ for all $t\in [0,1]$ and all $i=1,\ldots, d$. Hence,
$$
\left|\prod_{i=1}^d g_i(t) - \prod_{i=1}^d g_i(1)\right|
=
\sum_{j=1}^d \left( \prod_{i=1}^{j-1} g_i(t) \right) |g_j(t)-g_j(1)| \left( \prod_{i=j+1}^d g_i(1) \right)
\leq
C_2(1-t)^{\frac12},
$$
for all $t \in [0,1]$.
Analogously to $I_1(\alpha)$, it follows that $|I_2(\alpha)|$ stays bounded as $\alpha \downarrow -1$.
Hence, the limit~\eqref{eq:limit_if_b_has_pole} reduces to
\begin{align*}
&\lim_{\alpha \downarrow -1} (\alpha +1) b(\alpha; \alpha_1, \ldots, \alpha_d)
=
\lim_{\alpha \downarrow -1} (\alpha +1) I_3(\alpha)
=
\left( \prod_{i=1}^d g_i(1) \right) \lim_{\alpha \downarrow -1} (\alpha +1)  \int_{0}^{1}(1-t^2)^{\frac{\alpha-1}{2}} \dd t\\
=&
\left( \prod_{i=1}^d \frac{1}{c_{\frac{\alpha_i-1}{2}}} \right) \lim_{\alpha \downarrow -1} (\alpha +1) \frac{\sqrt \pi \Gamma \left( \frac{\alpha+1}{2} \right)}{2 \Gamma \left( \frac{\alpha+2}{2} \right)}
=
\prod_{i=1}^d \frac{1}{c_{\frac{\alpha_i-1}{2}}},
\end{align*}
completing the proof of~\eqref{eq:limit_if_b_has_pole} in the case $d\geq 1$.
\end{proof}

\subsection{Random beta simplices}

The preceding results simplify considerably in the case $n=d+1$, that is, for random simplices. We record the resulting formulas in the following  two corollaries.

\begin{corollary}[$n=d+1$: expected beta integrals of random beta simplices]\label{cor:expected_integral_simplex}
Let $d \ge 2$, and let $X_1,\ldots,X_{d+1}$ be independent random points with distributions $X_i \sim f_{d,\beta_i}$, where $\beta_1,\ldots,\beta_{d+1} \ge -1$. Put $\gamma_i \coloneqq \beta_i + \frac d2$ for $i=1,\ldots,d+1$, and consider the random simplex $\sS \coloneqq [X_1,\ldots,X_{d+1}]$.
\begin{enumerate}
\item[(i)] For every $\beta \in \CC$ with $\Re \beta > -\tfrac{d+1}{2}$ and $\beta \notin \{-1,-2,\ldots\}$, one has
{\small
\begin{multline*}
\E \left[ \int_\sS \left(1-\|x\|^2\right)^{\beta} \,\dd x \right]
=
\frac{2\pi^{d/2}\Gamma(\beta+1)}{\Gamma\!\left(\frac d2+\beta+1\right)}
\Theta\left(\beta+\frac d2;\{\gamma_1,\ldots,\gamma_{d+1}\};\varnothing\right)
\\
=
\frac{2\Gamma(\beta+1)}{\pi\,\Gamma\!\left(\frac d2+\beta+1\right)}
\frac{\Gamma\!\left(\frac d2+\beta+\sum_{i=1}^{d+1}\gamma_i+\frac32\right)}
{\Gamma\!\left(\frac d2+\beta+\sum_{i=1}^{d+1}\gamma_i+1\right)}
\left(\prod_{i=1}^{d+1}\frac{\Gamma(\gamma_i+1)}{\Gamma\!\left(\gamma_i+\frac12\right)}\right)
a\bigg(2\beta+d+2+\sum_{i=1}^{d+1}2\gamma_i;2\gamma_1,\ldots,2\gamma_{d+1}\bigg).
\end{multline*}
}
\item[(ii)] For every $\beta \in \{-1,-2,-3,\ldots\}$ satisfying $\beta > -\frac{d+1}{2}$, one has
{\small
\begin{multline*}
\E \left[ \int_\sS \left(1-\|x\|^2\right)^{\beta} \,\dd x \right]
=
\frac{4}{\pi\,\Gamma\!\left(\beta+\frac d2+1\right)}
\frac{(-1)^{-\beta-1}}{(-\beta-1)!}
\frac{\Gamma\!\left(\frac d2+\beta+\sum_{i=1}^{d+1}\gamma_i+\frac32\right)}
{\Gamma\!\left(\frac d2+\beta+\sum_{i=1}^{d+1}\gamma_i+1\right)}
\\
\times
\left(\prod_{i=1}^{d+1}\frac{\Gamma(\gamma_i+1)}{\Gamma\!\left(\gamma_i+\frac12\right)}\right)
a'\bigg(2\beta+d+2+\sum_{i=1}^{d+1}2\gamma_i;2\gamma_1,\ldots,2\gamma_{d+1}\bigg).
\end{multline*}
}
\end{enumerate}
\end{corollary}

\begin{proof}
\textit{Proof of (i).}
This follows from Theorem~\ref{theo:expected_beta_volume_beta_polytopes}. Indeed, if $n=d+1$, then the sums in~\eqref{eq:beta_integral_for_hyp_vol_line_1} and~\eqref{eq:beta_integral_for_hyp_vol_line_2} contain only one summand, corresponding to $I=\{1,\ldots,d+1\}$. Moreover, by~\eqref{eq:b_for_d_0_and_1},
\[
\bigg(2\beta+d+1+\sum_{i=1}^{d+1}2\gamma_i\bigg)
b\bigg(2\beta+d+\sum_{i=1}^{d+1}2\gamma_i;\varnothing\bigg)
=
\frac{2\sqrt{\pi}\,\Gamma\!\left(\beta+\frac d2+\sum_{i=1}^{d+1}\gamma_i+\frac32\right)}
{\Gamma\!\left(\beta+\frac d2+\sum_{i=1}^{d+1}\gamma_i+1\right)}.
\]

\medskip
\noindent
\textit{Proof of (ii).}
This follows from Theorem~\ref{theo:expected_hyp_volume_beta_polytopes_at_poles}. Again, for $n=d+1$, the sum in~\eqref{eq:beta_integral_for_hyp_vol_poles} contains only one summand, namely the one corresponding to $I=\{1,\ldots,d+1\}$, and the same simplification of the $b$-term as in part~(i) applies.
\end{proof}

\begin{corollary}[$n=d+1$: expected hyperbolic volume of random beta simplices]\label{cor:expected_hyp_volume_beta_simplex}
Let $d \ge 2$, and let $X_1,\ldots,X_{d+1}$ be independent random points with distributions $X_i \sim f_{d,\beta_i}$, where $\beta_1,\ldots,\beta_{d+1} \ge -1$. Put $\gamma_i \coloneqq \beta_i + \frac d2$ for $i=1,\ldots,d+1$, and consider the random simplex $\sS \coloneqq [X_1,\ldots,X_{d+1}]$.
\begin{enumerate}
\item[(i)] If $d$ is \textbf{even}, then
{\small
\begin{align}
\E \Vol_d^{\mathrm{hyp}}(\sS)
&=
\frac{2(-2\pi)^{d/2}}{(d-1)!!}\,
\Theta\left(-\frac12;\{\gamma_1,\ldots,\gamma_{d+1}\};\varnothing\right)
\label{eq:beta_integral_for_hyp_vol_hyperbolic_simplex_even_as_theta}
\\
&=
\frac{2(-2)^{d/2}}{\pi(d-1)!!}
\left(\prod_{i=1}^{d+1}\frac{\Gamma(\gamma_i+1)}{\Gamma\!\left(\gamma_i+\frac12\right)}\right)
\frac{\Gamma\!\left(1+\sum_{i=1}^{d+1}\gamma_i\right)}
{\Gamma\!\left(\frac12+\sum_{i=1}^{d+1}\gamma_i\right)}
a\bigg(1+\sum_{i=1}^{d+1}2\gamma_i;2\gamma_1,\ldots,2\gamma_{d+1}\bigg).
\label{eq:beta_integral_for_hyp_vol_hyperbolic_simplex_even}
\end{align}
}
\item[(ii)] If $d$ is \textbf{odd}, then
{\small
\begin{multline}\label{eq:beta_integral_for_hyp_vol_hyperbolic_simplex_odd}
\E \Vol_d^{\mathrm{hyp}}(\sS)
=
\frac{4(-1)^{\frac{d-1}{2}}}{\pi^{3/2}\left(\frac{d-1}{2}\right)!}
\left(\prod_{i=1}^{d+1}\frac{\Gamma(\gamma_i+1)}{\Gamma\!\left(\gamma_i+\frac12\right)}\right)
\\
\times
\frac{\Gamma\!\left(1+\sum_{i=1}^{d+1}\gamma_i\right)}
{\Gamma\!\left(\frac12+\sum_{i=1}^{d+1}\gamma_i\right)}
a'\bigg(1+\sum_{i=1}^{d+1}2\gamma_i;2\gamma_1,\ldots,2\gamma_{d+1}\bigg).
\end{multline}
}
\end{enumerate}
\end{corollary}

\begin{proof}
The proof follows the same pattern as that of Corollary~\ref{cor:expected_integral_simplex}, now using the corresponding formulas for the expected hyperbolic volume.

\medskip
\noindent
\textit{Proof of (i).}
This follows from Theorem~\ref{theo:expected_hyp_volume_beta_polytopes_even_d}. The identity~\eqref{eq:beta_integral_for_hyp_vol_hyperbolic_simplex_even_as_theta} is immediate from~\eqref{eq:beta_integral_for_hyp_vol_line_1_hyperbolic}, because for $n=d+1$ the only admissible subset is $I=\{1,\ldots,d+1\}$.

To obtain~\eqref{eq:beta_integral_for_hyp_vol_hyperbolic_simplex_even}, we specialize~\eqref{eq:beta_integral_for_hyp_vol_line_2_hyperbolic} to $n=d+1$, so that again only the summand with $I=\{1,\ldots,d+1\}$ remains. If
$
\sum_{i=1}^{d+1}2\gamma_i>0,
$
then the first argument of the $b$-term is greater than $-1$, and by~\eqref{eq:b_for_d_0_and_1},
\begin{equation}\label{eq:b_term_simplification}
\bigg(\sum_{i=1}^{d+1}2\gamma_i\bigg)
b\bigg(\sum_{i=1}^{d+1}2\gamma_i-1;\varnothing\bigg)
=
\frac{2\sqrt{\pi}\,\Gamma\!\left(\sum_{i=1}^{d+1}\gamma_i+1\right)}
{\Gamma\!\left(\sum_{i=1}^{d+1}\gamma_i+\frac12\right)}.
\end{equation}
Substituting this identity and simplifying the constants yields~\eqref{eq:beta_integral_for_hyp_vol_hyperbolic_simplex_even}.
It remains to consider the case
$
\sum_{i=1}^{d+1}2\gamma_i=0,
$
which is possible if and only if $d=2$ and $\beta_1=\beta_2=\beta_3=-1$. In this case, according to the convention in Theorem~\ref{theo:expected_hyp_volume_beta_polytopes_even_d} and Lemma~\ref{lemma:limit_if_b_has_pole}, one interprets
$
0\cdot b(-1;\varnothing)=2,
$
which is consistent with~\eqref{eq:b_term_simplification}. This completes the proof of part~(i).

\medskip
\noindent
\textit{Proof of (ii).}
Apply Theorem~\ref{theo:expected_hyp_volume_beta_polytopes_odd_d}, specifically~\eqref{eq:beta_integral_for_hyp_vol_poles_hyperbolic_odd_d}, with $n=d+1$. Then $I=\{1,\ldots,d+1\}$ is again the only admissible subset, and the same simplification of the $b$-term as in~\eqref{eq:b_term_simplification} yields~\eqref{eq:beta_integral_for_hyp_vol_hyperbolic_simplex_odd}.
\end{proof}

\section{Examples}

In this section, we discuss several situations in which the functions $a$ and $b$ simplify.

\subsection{Random ideal polygons in dimension \texorpdfstring{$2$}{2}}

Let $X_1,\ldots,X_n$ be independent random points uniformly distributed on the unit circle $\partial \mathbb{B}^2$, that is, $X_i \sim f_{2,-1}$ for all $i=1,\ldots,n$. Applying~\eqref{eq:beta_integral_for_hyp_vol_2_line_2_hyperbolic} with $d=2$ and
$
\gamma_i=\beta_i+\frac d2=-1+\frac22=0
$ for all  $i=1,\ldots,n$,
we obtain
\begin{align*}
\E \Vol_2^{\mathrm{hyp}}([X_1,\ldots,X_n])
&=
\frac{-2\pi}{\pi}
\left(
\pi
-
\frac{1}{\pi^n}\, n\, a(1;0)\cdot 0 \cdot b_{n-1}(-1;0)
\right)
=
(n-2)\pi.
\end{align*}
Here we interpret the factor $0\cdot b_{n-1}(-1;0)$ as $\pi^{n-1}$ using Lemma~\ref{lemma:limit_if_b_has_pole}, and we use
$
a(1;0)=\frac{\pi^2}{2},
$
which follows from~\eqref{eq:a_for_d_0_and_1}.

This agrees with the classical area formula in the hyperbolic plane: if a hyperbolic $n$-gon has interior angles $\alpha_1,\ldots,\alpha_n$, then its area is
$
(n-2)\pi-\sum_{i=1}^n \alpha_i.
$
In particular, every ideal hyperbolic $n$-gon has area $(n-2)\pi$, and therefore
\[
\Vol_2^{\mathrm{hyp}}([X_1,\ldots,X_n])=(n-2)\pi
\qquad \text{a.s.}
\]

\subsection{Random ideal polytopes in dimension \texorpdfstring{$3$}{3}}

\begin{theorem}[Expected hyperbolic volume of a random ideal polytope in dimension $3$]\label{theo:exp_hyp_vol_random_ideal_polytope_dim_3}
Let $n \ge 4$, and let $X_1,\ldots,X_n$ be independent random points uniformly distributed on the unit sphere $\partial \mathbb{B}^3 \subset \mathbb{R}^3$. Then
\begin{align}\label{eq:exp_hyp_vol_random_ideal_polytope_dim_3}
\E \Vol_3^{\mathrm{hyp}}([X_1,\ldots,X_n])
=
\pi\left(\frac n2 - \sum_{j=1}^{n-1} \frac{1}{j}\right).
\end{align}
\end{theorem}

\begin{example}
For $n=4$, the expected hyperbolic volume of the random ideal tetrahedron $[X_1,\ldots,X_4]$ in hyperbolic $3$-space is $\pi/6$. The next few values of~\eqref{eq:exp_hyp_vol_random_ideal_polytope_dim_3} are
\[
\frac{5\pi}{12} \text{ for } n=5, \qquad
\frac{43\pi}{60} \text{ for } n=6, \qquad
\frac{21\pi}{20} \text{ for } n=7, \qquad
\text{and} \qquad
\frac{197\pi}{140} \text{ for } n=8.
\]
\end{example}


To prove Theorem~\ref{theo:exp_hyp_vol_random_ideal_polytope_dim_3}, we apply Theorem~\ref{theo:expected_hyp_volume_beta_polytopes_odd_d} with $d=3$ and $\beta_1=\cdots=\beta_n=-1$. After simplification, this yields
\begin{equation}\label{eq:vol_3_hyp_ideal_poly_proof_1}
\E \Vol_3^{\mathrm{hyp}}([X_1,\ldots,X_n])
=
-2^{1-n}
\sum_{\substack{k=4,6,\ldots\\ k\le n}}
\binom{n}{k}\,
a_k'(k+1;1)\cdot k \cdot b_{n-k}(k-1;1).
\end{equation}
In the next two lemmas, we evaluate $a_d(\alpha;1)$ and $b_d(\alpha;1)$. We then compute $a_k'(k+1;1)$ for $k=2,4,6,\ldots$.

\begin{lemma}\label{lemma:a_with_1}
For $d \in \mathbb{N}_0$ and $\alpha>d$, one has
\[
a_d(\alpha;1)
\equiv
a(\alpha;\underbrace{1,\ldots,1}_{d\text{ times}})
=
\frac{\pi\,\Gamma(\alpha-d)}
{2^{\alpha-d-1}\,\Gamma\!\left(\frac{\alpha+1}{2}\right)\Gamma\!\left(\frac{\alpha+1}{2}-d\right)}.
\]
\end{lemma}
\begin{proof}
Starting from Definition~\ref{def:a_and_b_quantities} of the function $a$ and using the identity
\[
F_1(\ii x)
=
\frac{1}{2c_0} + \ii \int_0^x \cosh y\,\dd y
=
1+\ii \sinh x,
\]
we obtain
\begin{align*}
a_d(\alpha;1)
&=
\int_{-\infty}^{\infty} \cosh^{-\alpha} (x) \prod_{j=1}^d F_1(\ii x) \dd x
=
\int_{-\infty}^{\infty} \cosh^{-\alpha} (x) (1+\ii \sinh x)^d \dd x\\
&=
\int_{-\infty}^{\infty} (1+u^2)^{-\alpha/2} (1+\ii u)^d \frac{\dd u}{\sqrt{1+u^2}}
=
\int_{-\infty}^{\infty} (1+\ii u)^d (1+u^2)^{-\frac{\alpha+1}2} \dd u \\
&=
\int_{-\pi/2}^{\pi/2} \left(\eee^{\ii \theta} \frac1{\cos \theta} \right)^d \cos^{\alpha+1}(\theta) \frac{\dd \theta}{\cos^2 \theta}
=
\int_{-\pi/2}^{\pi/2} \eee^{\ii d \theta} \cos^{\alpha-1-d} (\theta) \dd \theta,
\end{align*}
with the substitution $u = \sinh x$ and then $u=\tan \theta$. The imaginary part of the latter integrand is an odd function, while the real part is even. It follows that
\begin{align*}
a_d(\alpha;1)
&=
2 \int_0^{\pi/2} \cos( d \theta) \cos^{\alpha-1-d} (\theta) \dd \theta.
\end{align*}
This is a standard integral; applying~\cite[Formula 3.632.5]{gradshteyn_ryzhik_book} with $p=\tfrac{\alpha+1}2$ and $q=\tfrac{\alpha+1}2 -d$ gives
\begin{align*}
a_d(\alpha;1)
&=
2 \frac{\pi}{2^{\alpha-d} (\alpha-d) B\left( \frac{\alpha+1}2, \frac{\alpha+1}2-d \right)},
\end{align*}
where $B(\cdot, \cdot)$ denotes the Euler beta function. Rewriting this in terms of the gamma function completes the proof.
\end{proof}

\begin{lemma}\label{lemma:b_with_1}
For $d\in \N_0$ and $\alpha>-1$,
$$
b_d(\alpha;1)
\equiv
b(\alpha;\underbrace{1,\ldots, 1}_{d \text{ times}})
=
\frac{2^{\alpha+d} \Gamma\left( \frac{\alpha+1}2 \right) \Gamma\left( \frac{\alpha+1}2 +d \right) }{\Gamma(\alpha+d+1)}.
$$
\end{lemma}
\begin{proof}
Definition~\ref{def:a_and_b_quantities} of the function $b$ entails
\begin{align*}
b_d(\alpha;1)
&=
\int_{-\pi/2}^{\pi/2} \cos^\alpha(x) \prod_{j=1}^d F_1(x) \dd x
=
\int_{-\pi/2}^{\pi/2} \cos^\alpha(x) (1+\sin x)^d \dd x\\
&=
\int_{-1}^{+1} (1+u)^d (1-u^2)^{\frac{\alpha-1}2} \dd u
=
2^{\alpha +d} B\left( \frac{\alpha+1}2, \frac{\alpha+1}2+d \right)
\end{align*}
by first substituting $u=\sin x$, then using $1-u^2=(1+u)(1-u)$.
Expressing  the Euler beta function $B(\,\cdot\, , \,\cdot \,)$ in terms of the gamma function completes the proof.
\end{proof}

\begin{lemma}\label{lemma:a_with_1_der}
For every $k=2,4,6,\ldots$ one has $a_k(k+1;1) = 0$ and $a_k'(k+1;1) =\frac{\pi}{k}(-1)^{\frac{k}{2}-1}$.
\end{lemma}
\begin{proof}
Set $\alpha=k+1+\varepsilon$ with $\eps\in \mathbb C$. Then Lemma~\ref{lemma:a_with_1} gives
\[
a_k(k+1+\varepsilon;1)
=
\pi\,2^{-\varepsilon}
\frac{\Gamma(1+\varepsilon)}
{\Gamma\!\left(\frac k2+1+\frac{\varepsilon}{2}\right)
 \Gamma\!\left(1-\frac k2+\frac{\varepsilon}{2}\right)}.
\]
The right-hand side is a meromorphic function of $\eps$. We compute its Laurent series at $\eps= 0$.
The first three factors are regular near \(\varepsilon=0\), and
\[
2^{-\varepsilon}=1+O(\varepsilon),\qquad
\Gamma(1+\varepsilon)=1+O(\varepsilon),\qquad
\frac{1}{\Gamma\!\left(\frac k2+1+\frac{\varepsilon}{2}\right)}
=
\frac{1}{(k/2)!}+O(\varepsilon),
\]
as $\eps \to 0$. For the last factor, we recall that $k\in \{2,4,6,\ldots\}$, so that
\[
\frac{1}{\Gamma\!\left(1-\frac k2+\frac{\varepsilon}{2}\right)}
=
(-1)^{\frac k2-1}\left(\frac k2-1\right)!\,\frac{\varepsilon}{2}
+O(\varepsilon^2).
\]
Multiplying the expansions gives
\[
a_k(k+1+\varepsilon;1)
=
\pi\Bigl(1+O(\varepsilon)\Bigr)
\Bigl(\frac{1}{(k/2)!}+O(\varepsilon)\Bigr)
\Bigl((-1)^{\frac k2-1}\left(\frac k2-1\right)!\,\frac{\varepsilon}{2}+O(\varepsilon^2)\Bigr)
=
\frac{\pi}{k}(-1)^{\frac k2-1}\varepsilon
+O(\varepsilon^2).
\]
We conclude that  $a_k(k+1;1) = 0$ and $a_k'(k+1;1) =\frac{\pi}{k}(-1)^{\frac{k}{2}-1}$.
\end{proof}

In view of Lemmas~\ref{lemma:b_with_1} and~\ref{lemma:a_with_1_der}, Equation~\eqref{eq:vol_3_hyp_ideal_poly_proof_1} takes the form
$$
\E \Vol_3^{\mathrm{hyp}} ([X_1, \ldots, X_n])
=
\pi \sum_{\substack{k=4,6,\ldots\\ k\leq n}} (-1)^{k/2} \binom{n}{k}
\frac{\Gamma\left( \frac k2 \right) \Gamma\left(n - \frac k2 \right)}{\Gamma\left( n \right)}.
$$
The next lemma simplifies the right-hand side and completes the proof of Theorem~\ref{theo:exp_hyp_vol_random_ideal_polytope_dim_3}.

\begin{lemma}\label{lemma:alternating_sum_harmonic}
For every \(n\geq 2\),
$$\sum_{\ell = 2}^{\left\lfloor \frac n2 \right\rfloor} (-1)^\ell \binom{n}{2\ell}
\frac{(\ell-1)! (n-\ell-1)!}{(n-1)!}
=
\frac n2 - \sum_{j=1}^{n-1}\frac1j.
$$
\end{lemma}

\begin{proof}
For $n=2,3$, the equation is evident. Let $n\geq 4$. The proof is via a Wilf--Zeilberger pair; see~\cite[Chapter 7]{petkovsek_wilf_zeilberger}. Let us write the sum as
$$
S_n
:=
\sum_{\ell=2}^{\lfloor n/2\rfloor} F(n,\ell),
\qquad
F(n,\ell):=
(-1)^\ell \binom{n}{2\ell}\frac{(\ell-1)!(n-\ell-1)!}{(n-1)!},
$$
for all $n\geq 4$. Using the convention that $\binom{a}{b}=0$ if $a<b$, it follows that $F(n,\ell)=0$ for $\ell > \lfloor \tfrac n2 \rfloor$, and hence
$$
S_{n+1}-S_n
=
\sum_{\ell=2}^{\lfloor \frac{n+1}{2}\rfloor} \left( F(n+1,\ell) - F(n,\ell) \right).
$$
For $n \geq 4$ and $\ell \geq 2$, set
$$
G(n,\ell):=
-(-1)^\ell \binom{n-1}{2\ell-2}\frac{(\ell-1)!(n-\ell)!}{n!}.
$$
A computation shows that $F$ and $G$ form  a Wilf--Zeilberger pair, that is
$$
F(n+1,\ell) - F(n,\ell)
=
G(n,\ell+1) - G(n,\ell)
$$
for all $\ell\in \{2,\ldots, \lfloor \tfrac{n+1}{2} \rfloor\}$. Hence, for all $n\geq 4$,
\begin{align*}
S_{n+1}-S_n
&=
\sum_{\ell=2}^{\lfloor \frac{n+1}{2} \rfloor} \left( F(n+1,\ell) - F(n,\ell) \right)
=
\sum_{\ell=2}^{\lfloor \frac{n+1}{2} \rfloor} \left( G(n,\ell+1) - G(n,\ell) \right)\\
&=
G\left( n, \left\lfloor \frac{n+1}{2} \right\rfloor +1 \right) - G(n,2)
=
0+
\frac{n-2}{2n}
=
\frac 12 - \frac 1n.
\end{align*}
Therefore, for every $n\geq 4$,
\begin{align*}
S_n
=
S_4+\sum_{j=4}^{n-1}(S_{j+1}-S_j)
=
\frac16+\sum_{j=4}^{n-1}\left(\frac12-\frac1j\right)
=
\frac n2-\sum_{j=1}^{n-1}\frac1j,
\end{align*}
and the proof is complete.
\end{proof}

\subsection{Random ideal simplices in arbitrary dimension}

We now specialize our general formulas to the case of random ideal simplices.

\begin{theorem}[Expected hyperbolic volume of a random ideal simplex]\label{theo:expected_hyp_volume_random_ideal_simplex}
Let $d \ge 2$ be an integer, let $X_1,\ldots,X_{d+1}$ be independent random points uniformly distributed on the unit sphere $\partial \mathbb{B}^d \subset \mathbb{R}^d$, and let $\sS \coloneqq [X_1,\ldots,X_{d+1}]$ be their convex hull.
\begin{enumerate}
\item[(i)] If $d$ is \textbf{even}, then
\begin{align}
\E \Vol_d^{\mathrm{hyp}}(\sS)
&=
\frac{2^{1+d/2}}{\pi(d-1)!!}
\frac{\Gamma\left(\frac d2\right)^{d+1}}{\Gamma\left(\frac{d-1}2\right)^{d+1}}
\frac{\Gamma\left(\frac{d(d-1)}2\right)}{\Gamma\left(\frac{d(d-1)-1}2\right)}
\int_{-\infty}^{\infty}
\frac{\left(\int_0^t \sinh^{d-2}(u)\,\dd u\right)^{d+1}}
{(\sinh t)^{d(d-1)-1}}\,\dd t.
\label{eq:exp_volume_hyp_ideal_simplex_even_dimension}
\end{align}

\item[(ii)] If $d$ is \textbf{odd}, then
\begin{equation}\label{eq:exp_volume_hyp_ideal_simplex_odd_dimension}
\E \Vol_d^{\mathrm{hyp}}(\sS)
=
\frac{
2\pi^{\frac{d-1}{2}}
}{
\left(\frac{d-1}{2}\right)!
\binom{d^2-d-2}{\frac{d^2-d-2}{2}}
}
\int_0^1
t^{\frac{d-1}{2}}
\left(
\sum_{j=0}^{\frac{d-3}{2}}
(-1)^j
\binom{d-2}{\frac{d-3}{2}-j}
\, t^j
\right)^{d+1}
\,\dd t.
\end{equation}
\end{enumerate}
\end{theorem}

\begin{example}[Small dimensions]
	In small dimensions, Theorem~\ref{theo:expected_hyp_volume_random_ideal_simplex}  yields the following values of $V_d:= \E \Vol_d^{\mathrm{hyp}} (\sS)$:	
	\[
	V_2=\pi,\qquad V_3=\frac{\pi}{6}, \qquad V_4=\frac{4\pi^2}{3}-\frac{86528}{6615},\qquad V_5=\frac{943\,\pi^2}{942480},
	\]
	\[
	V_6=
\frac{34}{15}\,\pi^3
-
\frac{1166172999537393664}{47992913336092725}\,\pi
+
\frac{7754705186848768}{407510816383125}\,\frac{1}{\pi},
\qquad
	V_7=\frac{6952469612009\,\pi^3}{2292117595080112800}.
	\]
\end{example}

\begin{remark}[$d\to\infty$]
Using the standard Laplace method, it can be shown (both for even and odd $d$) that
\[
\E \Vol_d^{\mathrm{hyp}} (\sS) \sim \frac{\eee^{5/4}}{\sqrt{\pi}}
\left(\frac{\sqrt \eee}{d}\right)^d
\qquad(d\to\infty).
\]
For comparison, \citet{haagerup_munkholm_simplices_max_vol_hyperbolic} have shown that the hyperbolic volume of the \emph{regular} ideal simplex in $\mathbb B^d$ is asymptotically equivalent to $\eee \sqrt d / d!$ as $d\to\infty$.
\end{remark}

\subsubsection{Proof of Theorem~\ref{theo:expected_hyp_volume_random_ideal_simplex}: even dimensions}
To evaluate the expected hyperbolic volume of $\sS$, we take  $\beta_1 = \ldots = \beta_{d+1} =  -1$ in Corollary~\ref{cor:expected_hyp_volume_beta_simplex}.
For even $d$, this gives
\begin{align}\label{eq:pf_hyp_vol_simplex_odd_dim_first_step}
\E \Vol_d^{\mathrm{hyp}} (\sS)
=&
\frac{2(-2)^{d/2}}{\pi (d-1)!!}
\frac{\Gamma\left( \frac d2 \right)^{d+1}}{\Gamma\left( \frac {d-1}2 \right)^{d+1}} \frac{\Gamma\left( \frac{d(d-1)}2 \right)}
{\Gamma\left( \frac{d(d-1)-1}2 \right)}
a_{d+1}\left(d(d-1)-1;d-2\right).
\end{align}
Definition~\ref{def:a_and_b_quantities} of the function $a$ yields
$$
a_{d+1}\left(d(d-1)-1;d-2\right)
=
\int_{-\infty}^{\infty} g(x) \, \dd x,
\qquad
g(x)
:=
\frac{\left(\frac{1}{2 c_{(d-3)/2}} +\ii \int_0^x \cosh^{d-2}(y)\,\dd y\right)^{d+1}}{(\cosh x)^{d(d-1)-1}}.
$$
Let us show that the line of integration can be shifted by $\ii \tfrac\pi2$ without changing the value, i.e.\
\begin{equation}\label{eq:ideal_hyper_simpl_cauchy_shift}
\int_{-\infty}^{\infty} g(x) \dd x
=
\int_{-\infty}^{\infty}g\!\left(x+\ii \frac{\pi}{2}\right)\dd x.
\end{equation}
The first step is to show that the integrand $g$ is holomorphic on the strip $\mathcal D := \{z\in\CC: -\delta < \Im z < \pi/2+ \delta\}$, where $\delta>0$ is sufficiently small, so we can then apply Cauchy's theorem.
Indeed, the function $g$ is meromorphic. Moreover, the only possible singularity of
$g$ in the closed strip $\mathcal D$ occurs at $x=\ii\pi/2$, where $\cosh x=0$.
We first show that this singularity is removable. Indeed,
\begin{equation}\label{eq:pf_hyp_vol_simplex_odd_dim_integral_identity}
\frac{1}{2c_{(d-3)/2}}+\ii\int_0^z \cosh^{d-2}(y)\,\dd y
=
\ii\int_{\ii\pi/2}^z \cosh^{d-2}(y)\,\dd y = O\!\left((z-\ii\pi/2)^{d-1}\right)
\end{equation}
as $z\to \ii\pi/2$, because $\cosh y$ has a simple zero at $y=\ii\pi/2$.
It follows that
$$
\left(\frac{1}{2c_{(d-3)/2}} +\ii\int_0^z \cosh^{d-2}(y)\,\dd y \right)^{d+1}
=
O\!\left((z-\ii\pi/2)^{(d-1)(d+1)}\right),
$$
while on the other hand
$
\cosh^{d(d-1)-1}(z)
\sim
C (z-\ii\pi/2)^{d(d-1)-1}.
$
Together, this gives $g(z)=O((z-\ii\pi/2)^d)$ as $z\to \ii\pi/2$, so $g$ extends holomorphically to  \(z=\ii\pi/2\). Hence $g$ is holomorphic on the whole strip $\{z\in\CC: -\delta \leq \Im z\leq \pi/2+ \delta\}$.

To justify the use of the Cauchy theorem and prove~\eqref{eq:ideal_hyper_simpl_cauchy_shift}, it suffices to check that for all sufficiently large positive $R$,
\begin{equation}\label{eq:g_exp_decay_to_justify_shifting_contour}
\sup_{0\le t\le \pi/2}|g(\pm R+\ii t)|
\leq
C \eee^{-R}.
\end{equation}
There exist constants \(C_1,C_2>0\) such that
$
C_1 \eee^{|\sigma|}
\leq
|\cosh(\sigma+\ii \tau)|
\leq
C_2 \eee^{|\sigma|}
$
for all sufficiently large $|\sigma|$, uniformly in $\tau\in[0,\pi/2]$. Moreover, integrating first horizontally and then vertically, there exists a constant $C_3>0$ such that
$$
\left|\int_0^{\sigma+\ii\tau}\cosh^{d-2}(y)\, \dd y\right|
\leq
C_3 \eee^{(d-2)|\sigma|}
$$
for all sufficiently large $|\sigma|$, again uniformly in $\tau\in[0,\pi/2]$. Therefore, for some constant $C_4>0$,
$$
|g(\sigma+\ii\tau)|
\leq
C_4 \eee^{-(d(d-1)-1)|\sigma|} \eee^{(d+1)(d-2)|\sigma|}
=
C_4 \eee^{-|\sigma|}.
$$
This proves~\eqref{eq:g_exp_decay_to_justify_shifting_contour} and thus~\eqref{eq:ideal_hyper_simpl_cauchy_shift}.

Substituting $x\mapsto x + \ii\frac{\pi}{2}$ in  the definition of $g$ and using~\eqref{eq:pf_hyp_vol_simplex_odd_dim_integral_identity} gives
\begin{align*}
	g\left(x+\ii\frac{\pi}{2}\right)
	=
	\frac{\left( \ii \int_{\ii \pi/2}^{x+\ii\pi/2}\cosh^{d-2}(y)\, \dd y \right)^{d+1}}{{(\ii\sinh x)^{d(d-1)-1}}}
	=
	(-1)^{d/2} \frac{\left(\int_0^x \sinh^{d-2}(u)\, \dd u \right)^{d+1}}{(\sinh x)^{d(d-1)-1}}.
\end{align*}
Thus,
$$
a_{d+1}\left(d(d-1)-1;d-2\right)
=
\int_{-\infty}^{\infty}g\!\left(x+\ii\frac{\pi}{2}\right) \dd x
=
(-1)^{d/2} \int_{-\infty}^{\infty}  \frac{\left(\int_0^x \sinh^{d-2}(u)\, \dd u \right)^{d+1}}{(\sinh x)^{d(d-1)-1}} \, \dd x.
$$
Plugging this into~\eqref{eq:pf_hyp_vol_simplex_odd_dim_first_step} completes the proof.

\subsubsection{Proof of Theorem~\ref{theo:expected_hyp_volume_random_ideal_simplex}: odd dimensions}
Let us sketch the main steps of the proof. The beta distribution $f_{d,-1}$ is the uniform distribution on the $(d-1)$-dimensional unit sphere.  Corollary~\ref{cor:expected_hyp_volume_beta_simplex} (ii) with $\beta_1 = \ldots = \beta_{d+1} = -1$ gives
$$
\E\Vol_d^{\mathrm{hyp}}(\sS)
=
\frac{4 (-1)^{\frac{d-1}{2}}}{\pi^{3/2} \left( \frac{d-1}2\right)!} \left(\frac{\Gamma(\frac d2)}{\Gamma(\frac{d-1}{2})}\right)^{d+1}
\frac{\Gamma\!\left(\frac{d^2-d}{2}\right)} {\Gamma\!\left(\frac{d^2-d-1}{2}\right)}
a'_{d+1}(d^2-d-1;d-2).
$$
Hence the main task is to evaluate the derivative of
$$
a_{d+1}(\alpha;d-2)
=
\int_{-\infty}^{\infty} \cosh^{-\alpha}(u) F_{d-2}(\ii u)^{d+1} \dd u
$$
at $\alpha = d^2-d-1$. The first step is to simplify the term $F_{d-2}(\ii u)$.  This is done, in a slightly more general form, in Lemma~\ref{lemma:F_odd_on_imaginary_axis}. The proof of Theorem~\ref{theo:expected_hyp_volume_random_ideal_simplex} for odd dimension $d$ proceeds, as the proof of Lemma~\ref{lemma:a_with_1}, by the substitution $\tan \theta = \sinh u$. Differentiating with respect to $\alpha$ then yields the additional factor $\log \cos \theta$, and Lemma~\ref{lemma:poly_log_cos} provides a way to deal with the resulting integral. This leads to an explicit formula for
$
a'_{d+1}(d^2-d-1;d-2),
$
which is established, again in a slightly more general form, in Lemma~\ref{lemma:a_prime_repeated_odd_parameters}.

Let $B(\,\cdot\, , \,\cdot\,)$ denote the Euler beta function.  The incomplete beta function is defined by
$$
B_z(a,b)
:=
\int_0^z t^{a-1}(1-t)^{b-1}\,\dd t,
\qquad \text{ for } z\in [0,1), a>0, b>0.
$$

\begin{lemma}\label{lemma:F_odd_on_imaginary_axis}
Let $m\in\{1,2,3,\dots\}$ and define
$$
P_m(z)
:=
\sum_{r=0}^{m-1}\binom{2m-1}{m-1-r}z^r.
$$
If $u\in\R$ and $\theta\in(-\pi/2,\pi/2)$ are related by $\tan\theta=\sinh u$, then
$$
F_{2m-1}(\ii u)
=
B(m,m)\frac{\eee^{\ii \theta}}{\cos^{2m-1}\theta}P_m(\eee^{2\ii \theta}).
$$
\end{lemma}
\begin{proof}
For $x\in[-\pi/2,\pi/2]$, the substitution $t=\frac{1+\sin y}{2}$ gives
$$
F_{2m-1}(x)
=
\int_{-\pi/2}^x \cos^{2m-1}(y)\,\dd y
=
2^{2m-1} B_{\frac{1+\sin x}{2}}(m,m).
$$
For $z\in [0,1)$,~\cite[§ 8.17.5]{NIST:DLMF} says
\begin{equation*}\label{eq:incomplete_beta_integer}
B_z(m,m)
=
B(m,m) \sum_{r=0}^{m-1} \binom{2m-1}{m-1-r} z^{m+r}(1-z)^{m-1-r},
\end{equation*}
so it follows that
\begin{equation}\label{eq:pf_lemma_F_iu_and_theta}
F_{2m-1}(x)
=
2^{2m-1} B(m,m) \sum_{r=0}^{m-1} \binom{2m-1}{m-1-r} \left(\frac{1+\sin x}{2}\right)^{m+r} \left(\frac{1-\sin x}{2}\right)^{m-1-r}.
\end{equation}
Both sides are holomorphic functions of $x$ in $\CC \bsl ((-\infty, - \pi/2] \cup [\pi/2, +\infty))$; see Definition~\ref{def:F}). In particular the equality hence holds for $x=\ii u, u\in \R$.

Now let $u\in\R$ and $\theta\in(-\pi/2,\pi/2)$ satisfy $\tan\theta=\sinh u$. Then, $\sin(\ii u)=\ii \sinh u=\ii \tan\theta$, and therefore
$$
\frac{1+\sin(\ii u)}{2}
=
\frac{1+\ii\tan\theta}{2}
=
\frac{\eee^{\ii \theta}}{2\cos\theta},
\qquad
\frac{1-\sin(\ii u)}{2}
=
\frac{\eee^{-\ii\theta}}{2\cos\theta}.
$$
In~\eqref{eq:pf_lemma_F_iu_and_theta}, this yields
\begin{align*}
F_{2m-1}(\ii u)
&=
2^{2m-1} B(m,m) \sum_{r=0}^{m-1} \binom{2m-1}{m-1-r} \left(\frac{\eee^{\ii \theta}}{2\cos\theta}\right)^{m+r} \left(\frac{\eee^{-\ii\theta}}{2\cos\theta}\right)^{m-1-r}\\
&=
B(m,m)\frac{\eee^{\ii \theta}}{\cos^{2m-1}\theta} \sum_{r=0}^{m-1} \binom{2m-1}{m-1-r}\eee^{2\ii r\theta}=
B(m,m)\frac{\eee^{\ii \theta}}{\cos^{2m-1}\theta}P_m(\eee^{2\ii \theta}).
\end{align*}
\end{proof}

\begin{lemma}\label{lemma:poly_log_cos}
Let $q\in \{1,2,\ldots\}$ and let $Q$ be a polynomial. Then,
\begin{equation}\label{eq:poly_log_cos}
\int_{-\pi/2}^{\pi/2}
\eee^{2\ii q\theta} Q(\eee^{2\ii \theta}) \log(\cos\theta) \dd\theta
=
\frac{(-1)^{q+1}\pi}{2}
\int_0^1 t^{q-1} Q(-t) \dd t.
\end{equation}
\end{lemma}

\begin{proof}
By linearity, it suffices to prove~\eqref{eq:poly_log_cos} for monomials $Q(x) = x^k$ with $k=0,1,2,\ldots$.

On the left-hand side this yields
\begin{align*}
\int_{-\pi/2}^{\pi/2}
\eee^{2\ii q\theta} Q(\eee^{2\ii \theta}) \log(\cos\theta)\dd\theta
=
\int_{-\pi/2}^{\pi/2} \eee^{2\ii (q+k)\theta} \log(\cos\theta) \dd \theta
=
\int_{-\pi/2}^{\pi/2} \cos(2(q+k)\theta) \log(\cos\theta) \dd \theta.
\end{align*}
Indeed, the imaginary part vanishes because $\sin(2m\theta)\log(\cos\theta)$ is an odd function. Now, rewrite this with the Fourier series
$$
\log \cos \theta
=
-\log 2 + \sum_{n=1}^{\infty} \frac{(-1)^{n+1}}{n} \cos(2n\theta),
\qquad -\frac{\pi}{2}<\theta<\frac{\pi}{2}.
$$
Since $q+k \geq 1$, a computation shows that, for $n\in \{1,2,3,\ldots\}$, we have
$$
\int_{-\pi/2}^{\pi/2} \cos(2(q+k)\theta) \dd \theta
=
0,
\qquad
\int_{-\pi/2}^{\pi/2} \cos(2(q+k)\theta) \cos(2n\theta) \dd \theta
=
\begin{cases}
0, &\text{ if } n\neq q+k,\\
\frac{\pi}2, &\text{ if } n = q+k.
\end{cases}
$$
Together, the left-hand side is
$$
\int_{-\pi/2}^{\pi/2}
\eee^{2\ii q\theta} Q(\eee^{2\ii \theta}) \log(\cos\theta)\dd\theta
=
\frac{(-1)^{q+k+1}\pi}{2(q+k)}.
$$
Finally, evaluating the right-hand side with $Q=x^k$, gives the same result:
$$
\frac{(-1)^{q+1}\pi}{2}
\int_0^1 t^{q-1} Q(-t) \dd t
=
\frac{(-1)^{q+k+1}\pi}{2} \int_0^1 t^{q+k-1} \dd t
=
\frac{(-1)^{q+k+1}\pi}{2(q+k)}.
$$
\end{proof}

\begin{lemma}\label{lemma:a_prime_repeated_odd_parameters}
Let $m,q\in \{1,2,3,\ldots \}$.
Then,
\begin{equation}\label{eq:a_prime_repeated_odd_parameters}
a'_{2q}(2q(2m-1)+1;2m-1)
=
B(m,m)^{2q}\frac{(-1)^{q+1}\pi}{2}
\int_0^1 t^{q-1} \left( \sum_{r=0}^{m-1}(-1)^r \binom{2m-1}{m-1-r}t^r  \right)^{2q} \dd t.
\end{equation}
\end{lemma}
\begin{proof}
By definition,
$$
a_{2q}(\alpha;2m-1)
=
\int_{-\infty}^{\infty} \cosh^{-\alpha}(u) F_{2m-1}(\ii u)^{2q} \dd u.
$$
Now, substitute $\tan\theta=\sinh u$, $\theta \in (-\pi/2, \pi/2)$. For this, apply Lemma~\ref{lemma:F_odd_on_imaginary_axis}, recalling the definition
$$
P_m(z)
:=
\sum_{r=0}^{m-1}\binom{2m-1}{m-1-r}z^r,
$$
and note that
$\cosh u=\tfrac1{\cos \theta}$ and $\dd u=\tfrac{\dd\theta}{\cos \theta}$. Hence, we obtain
$$
a_{2q}(\alpha;2m-1)
=
B(m,m)^{2q} \int_{-\pi/2}^{\pi/2} \eee^{2\ii q\theta}P_m(\eee^{2\ii \theta})^{2q} \cos^{\alpha-2q(2m-1)-1}\theta \dd\theta.
$$
Differentiating with respect to $\alpha$ and evaluating at $\alpha=2q(2m-1)+1$ gives
$$
a'_{2q}(2q(2m-1)+1;2m-1)
=
B(m,m)^{2q} \int_{-\pi/2}^{\pi/2} \eee^{2\ii q\theta} P_m(\eee^{2\ii \theta})^{2q} \log(\cos\theta) \dd\theta.
$$
Now apply Lemma~\ref{lemma:poly_log_cos} with
$
Q(z):=P_m(z)^{2q}.
$
This yields
$$
a'_{2q}(2q(2m-1)+1;2m-1)
=
B(m,m)^{2q}\frac{(-1)^{q+1}\pi}{2} \int_0^1 t^{q-1}P_m(-t)^{2q} \dd t.
$$
This completes the proof.
\end{proof}

We are now ready to prove Theorem~\ref{theo:expected_hyp_volume_random_ideal_simplex} in case of odd dimension.

\begin{proof}[Proof of Theorem~\ref{theo:expected_hyp_volume_random_ideal_simplex} (odd dimension)]
Since the beta distribution $f_{d,-1}$ is the uniform distribution on the $(d-1)$-dimensional unit sphere, we apply~\eqref{eq:beta_integral_for_hyp_vol_hyperbolic_simplex_odd} from Corollary~\ref{cor:expected_hyp_volume_beta_simplex} with $\beta_1 = \ldots = \beta_{d+1} = -1$ (i.e.\ $\gamma_1 = \ldots = \gamma_{d+1} = \tfrac d2 -1$).
This yields
\begin{multline*}
\E\Vol_d^{\mathrm{hyp}}([X_1,\ldots,X_{d+1}])
=
\frac{4 (-1)^{\frac{d-1}{2}}}{\pi^{3/2} \left( \frac{d-1}2\right)!} \left(\frac{\Gamma(\frac d2)}{\Gamma(\frac{d-1}{2})}\right)^{d+1}
\frac{\Gamma\!\left(\frac{d^2-d}{2}\right)} {\Gamma\!\left(\frac{d^2-d-1}{2}\right)}
a'_{d+1}(d^2-d-1;d-2).
\end{multline*}

Next, we apply Lemma~\ref{lemma:a_prime_repeated_odd_parameters} with
$
m=\frac{d-1}{2}
$
and
$
q=\frac{d+1}{2}
$
(since
$
2q(2m-1)+1
=
d^2-d-1
$),
in order to further obtain
\begin{align*}
&\E \Vol_d^{\mathrm{hyp}}([X_1,\ldots,X_{d+1}])
=
\frac{4 (-1)^{\frac{d-1}{2}}}{\pi^{3/2} \left( \frac{d-1}2\right)!} \left(\frac{\Gamma(\frac d2)}{\Gamma(\frac{d-1}{2})}\right)^{d+1}
\frac{\Gamma\!\left(\frac{d^2-d}{2}\right)} {\Gamma\!\left(\frac{d^2-d-1}{2}\right)}
\\
& \qquad\qquad\qquad
\times B\left(\frac{d-1}{2},\frac{d-1}{2}\right)^{d+1} \frac{(-1)^{\frac{d+1}{2}+1}\pi}{2} \int_0^1 t^{\frac{d-1}{2}} \left(\sum_{r=0}^{\frac{d-3}{2}}(-1)^r \binom{d-2}{\frac{d-3}{2}-r}t^r \right)^{d+1} \dd t.
\end{align*}
The gamma terms can be simplified using the formulas
$$
\frac{\Gamma\left(\frac{d^2-d}{2}\right)}{\Gamma\left(\frac{d^2-d-1}{2}\right)}
=
\frac{2^{d^2-d-2}}{\sqrt{\pi}\,\binom{d^2-d-2}{\frac{d^2-d-2}{2}}}
\quad \text{ and } \quad
\frac{\Gamma\left(\frac{d}{2}\right)}{\sqrt{\pi}\,\Gamma\left(\frac{d-1}{2}\right)} = \frac{2^{2-d}}{B\left(\frac{d-1}{2}, \frac{d-1}{2}\right)},
$$
which both follow from the Legendre duplication formula. This gives
\begin{align*}
\E \Vol_d^{\mathrm{hyp}}([X_1,\ldots,X_{d+1}])
=
\frac{2\pi^{\frac{d-1}{2}}}{\left(\frac{d-1}{2}\right)!\binom{d^2-d-2}{\frac{d^2-d-2}{2}}} \int_0^1 t^{\frac{d-1}{2}}
\left(\sum_{r=0}^{\frac{d-3}{2}} (-1)^r \binom{d-2}{\frac{d-3}{2}-r} \, t^r\right)^{d+1} \, \dd t,
\end{align*}
which is precisely the stated formula~\eqref{eq:exp_volume_hyp_ideal_simplex_odd_dimension}.
\end{proof}

\subsection{Random polygons with \texorpdfstring{$\beta=0$}{beta = 0} in dimension \texorpdfstring{$2$}{2}}

Finally, we determine the expected hyperbolic area of a random polygon generated by $n$ independent random points uniformly distributed in the unit disk $\mathbb{B}^2$, where the uniform distribution is understood with respect to the Euclidean area measure.

\begin{proposition}[Expected hyperbolic area for $\beta=0$ in dimension $2$]\label{prop:hyp_vol_for_beta_0_dim_2}
Let $n \ge 3$, and let $X_1,\ldots,X_n$ be independent random points uniformly distributed in the unit disk $\mathbb{B}^2$. Viewing this disk as the Klein model of the hyperbolic plane, the expected hyperbolic area of the random polygon $\sP \coloneqq [X_1,\ldots,X_n]$ is given by
\begin{align}
\E \Vol_2^{\mathrm{hyp}}(\sP)
&=
-2\pi + \frac{2^{n-1}n}{\pi^{n-2}}\, b_{n-1}(1;2)
\label{eq:hyp_vol_for_beta_0_dim_2_eq1}\\
&=
-2\pi + \frac{2^{n-1}n}{\pi^{n-2}}
\int_{-1}^{1}
\left(
\int_{-1}^t \sqrt{1-s^2}\,\dd s
\right)^{n-1}
\dd t.
\label{eq:hyp_vol_for_beta_0_dim_2_eq2}
\end{align}
\end{proposition}

\begin{proof}
Since the uniform distribution on $\mathbb{B}^2$ is the beta distribution with parameter $0$, we have $X_i \sim f_{2,0}$ for all $i=1,\ldots,n$. Thus we may apply Theorem~\ref{theo:expected_hyp_volume_beta_polytopes_even_d}, more precisely~\eqref{eq:beta_integral_for_hyp_vol_2_line_2_hyperbolic}, with $d=2$ and $\beta_1=\cdots=\beta_n=0$. In this case, $\gamma_i=1$ for all $i$, and therefore
\begin{align*}
\E \Vol_2^{\mathrm{hyp}}([X_1,\ldots,X_n])
&=
-2\left(
\pi
-
\left(\frac{2}{\pi}\right)^n
\, n \, a(3;2)\cdot 2 \cdot b_{n-1}(1;2)
\right).
\end{align*}
Here we used that, for $d=2$, the only admissible subsets in~\eqref{eq:beta_integral_for_hyp_vol_2_line_2_hyperbolic} are the singletons, and hence there are exactly $n$ summands.
Next, by~\eqref{eq:a_for_d_0_and_1},
$
a(3;2)=\frac{\pi^2}{8}.
$
Substituting this into the previous identity yields~\eqref{eq:hyp_vol_for_beta_0_dim_2_eq1}. Equation~\eqref{eq:hyp_vol_for_beta_0_dim_2_eq2}  follows from~\eqref{eq:b:formula_1}.
\end{proof}

\begin{example}[Small values of $n$]
For small values of $n$, the expected hyperbolic area $\E \Vol_2^{\mathrm{hyp}}(\sP)$ is given by
\begin{align*}
&\pi-\tfrac{128}{15\pi}
\quad \text{for } n=3,
\qquad
&&2\pi-\tfrac{256}{15\pi}
\quad \text{for } n=4,
\\
&3\pi-\tfrac{128}{3\pi}+\tfrac{5537792}{33075\pi^3}
\quad \text{for } n=5,
\qquad
&&4\pi-\tfrac{256}{3\pi}+\tfrac{5537792}{11025\pi^3}
\quad \text{for } n=6.
\end{align*}
To compute these values, one may use the representation
\[
b_d(1;2)
=
\int_{-\pi/2}^{\pi/2} \cos x \, \bigl(F_2(x)\bigr)^d \,\dd x,
\qquad
F_2(x)
=
\int_{-\pi/2}^x \cos^2(y)\,\dd y
=
\frac{\pi}{4}+\frac{x}{2}+\frac12 \sin x \cos x.
\]
It follows that $b_d(1;2)$ can be expressed as a finite linear combination of integrals of the form $\int_{-\pi/2}^{\pi/2} x^a \sin^b x \cos^c x \,\dd x$, which are standard. 
\end{example}



\section*{Acknowledgement}
Supported by the German Research Foundation under Germany's Excellence Strategy  EXC 2044 / 2 -- 390685587, Mathematics M\"unster: Dynamics - Geometry - Structure and by the DFG priority program SPP 2265 Random Geometric Systems.  This paper benefited from many helpful and illuminating discussions with ChatGPT, which sharpened both the perspective and presentation.

\addcontentsline{toc}{section}{References}
\bibliography{expected_hyperbolic_volume_beta_polytopes_bib.bib}
\bibliographystyle{plainnat}

\vspace{1cm}

\footnotesize

\textsc{Zakhar Kabluchko: Institut f\"ur Mathematische Stochastik,
	Universit\"at M\"unster,
	Orl\'eans-Ring 10,
	48149 M\"unster, Germany}\\
\textit{E-mail}: \texttt{zakhar.kabluchko@uni-muenster.de}\\

\textsc{Philipp Schange: Institut f\"ur Mathematische Stochastik, Universit\"at M\"unster,
	Orl\'eans-Ring 10,
	48149 M\"unster, Germany}\\
\textit{E-mail}: \texttt{philipp.schange@uni-muenster.de}

\end{document}